\documentclass{article}
\usepackage[utf8]{inputenc}
\usepackage[main=english]{babel}
\usepackage{mathtools,amssymb,amsthm,amsfonts,upgreek,dsfont,mathabx}
\usepackage{geometry}
\usepackage{enumitem}
\usepackage{amsmath}
\usepackage{url}
\usepackage{fancybox}
\usepackage{fullpage}
\usepackage[colorlinks, linktocpage, citecolor=blue, linkcolor=blue]{hyperref}
\usepackage{multicol} 
\usepackage{array}
\usepackage{xargs}
\usepackage{float}
\usepackage[prependcaption]{todonotes}
\usepackage[capitalise]{cleveref}
\newcommandx{\Max}[2][1=]{\todo[inline, author={Maxime}, linecolor=green,backgroundcolor=green!25,bordercolor=green,#1]{#2}}

\newtheorem{theorem}{Theorem}[section]
\newtheorem{definition}{Definition}[section]

\newtheorem{lem}{Lemma}[section]

\renewcommand{\P}{\mathbb{P}}
\newcommand{\R}{\mathbb{R}}
\newcommand{\E}{\mathbb{E}}
\newcommand{\ind}{\mathds{1}}

\title{Local differential privacy in survival analysis using private failure indicators}

\author{
{\large Maxime Egéa}\footnote{Email: \texttt{maxime.egea@univ-angers.fr}}\\ \textit{LAREMA, Université d'Angers} 
\and
{\large Mikael Escobar-Bach}\footnote{Email: \texttt{mikael.escobar-bach@univ-angers.fr}.  \vspace*{.1cm}} \\ \textit{LAREMA, Université d'Angers}  
}

\begin{document}

\maketitle

\begin{abstract}
This work considers survival estimation with censored data under setups that preserve individual privacy. We provide an $\alpha$-locally differentially private mechanism on failure indicators and propose a non-parametric kernel estimator for the cumulative hazard function. Under mild conditions, we also prove lowers bounds on the minimax rates of convergence and show that our estimator is minimax optimal under well-chosen bandwidths. The method is illustrated with numerical results on synthetic data.
\end{abstract}

\smallskip
\noindent {\bf Key Words:} survival analysis; local differential privacy; right-censoring; minimax-optimality. 

\section{Introduction}
Censored data analysis is always a difficult challenge due the incompleteness nature of the observations. The censoring mechanism imposes a stringent setup for the statisticians and requires a dedicated methodology to obtain consistent statistical tools. Besides, studies in survival analysis usually apply to sensitive data where privacy protection is of crucial importance. In health care research or medicine, the release of shared databases has particularly increased the demand in guidelines for sanitized data. Differential privacy has prevailed as a strong candidate and provides a mathematical framework that helps in developing privacy-preserving methods \cite{Blu08,Dwo13}. In this framework, privacy mechanisms are considered as randomized algorithms that take an original database and produce a new set of random variables, from which the statistical analysis is solely based. Simultaneously, it is also important to properly define the privacy mechanism so that one can control the trade-off between the privacy protection and the statistical utility of the outputs. An algorithm is called differentially private if the change of at most one individual in the initial database only differs the likelihoods of the privatized databases by a small amount.\\ 
In a general manner, the randomization can be trusted to one common curator with a whole access to the raw data. However, scenarios in practice may require that each private sample are released by the data owners separately. This particular setup is referred to as local differential privacy and applies when privatized outputs are generated one at a time. Formally, an initial database $\{X_1,\ldots,X_n\}$ produces randomized data $\{Z_1\ldots,Z_n\}$ defined on a measurable space $(\mathcal{Z}^n,\mathcal{B}^n)$ such that $Z_i$ is generated accordingly to $X_i$ and $\{Z_1,\ldots,Z_{i-1}\}$. The random generation is described through a Markov kernel $Q:\mathcal{B}\times(\mathbb{R}^d\times\mathcal{Z}^{i-1})\to[0,1]$. The privacy mechanism is then called $\alpha$-locally differentially private for $\alpha>0$ if
\begin{eqnarray*}
    \sup_{B\in\mathcal{B}}\sup_{z_1,\ldots,z_{i-1}\in\mathcal{Z}}\sup_{x,x'\in\mathbb{R}^p}\dfrac{Q(B|X_i=x,Z_{i-1}=z_{i-1},\ldots,Z_1=z_1)}{Q(B|X_i=x',Z_{i-1}=z_{i-1},\ldots,Z_1=z_1)}\leq e^\alpha.   
\end{eqnarray*}
The parameter $\alpha>0$ controls the balance between privacy and statistical accuracy where privacy protection is strengthen as the parameter becomes smaller. Under this assumption, the random generation is interactive and each data provider can rely on previously transformed inputs to generate new outputs. However, many local privacy mechanisms are simpler and do not depend on some external information. A non-interactive privacy mechanism is then reduced to the same setting although $Z_i$ is now independent from the random variables $\{Z_1,\ldots,Z_{i-1}\}$ simplifying the previous condition to 
\begin{eqnarray}
\label{eq::ldp}
    \sup_{B\in\mathcal{B}}\sup_{x,x'\in\mathbb{R}^p}\dfrac{Q(B|X_i=x)}{Q(B|X_i=x')}\leq e^\alpha.   
\end{eqnarray}
 The literature on differential privacy is very active and more recent works with privacy purposes aim at developing machine learning algorithms; see \cite{Wan23} for a recent review, or to understand the statistical properties of the privacy mechanisms \cite{Nar23,Amo23}. Statistical utility of privacy mechanisms is usually based on the minimax framework and has debuted in \cite{Dwo10,Was10,Was12} before it has been applied to locally private procedures; see e.g \cite{Duc18,But21,Amo23}, and global ones \cite{Liu19}. In survival analysis, differential privacy is still at its early stage of development and many authors acknowledge the need for more attention in the future \cite{Fic21}. Although censoring prevents some of the desired data from getting recorded, there exists no value of $\alpha$ such that the inequality (\ref{eq::ldp}) is true if $Q$ describes the censoring mechanism. In a sense, censoring does not provide an adequate privacy protection from a differential privacy perspective. Recent works have attempted to propose privacy procedures for health databases \cite{Bon22} based on differential privacy, but most of the literature has explored the problem heuristically with time-to-event data \cite{OKe12,Bon22a} or connections between likelihood estimation and global privacy \cite{Ngu17}. In general, privacy in survival analysis focuses only on the output transformations of standard statistical methods, but does not propose privacy procedures for survival databases. Indeed, the relationships between survival outcomes make it difficult to consider any marginal resampling without altering model assumptions.\\ 
 In this paper, we study the estimation of the survival time distribution in the context of local differential privacy. We consider a regression model with independent censoring and propose to privatize the censoring indicators throughout an $\alpha$-locally differentially private mechanism. Although this setup is less restrictive than privatizing all the inputs, it still prevails any external observer to assess whether a survival time is censored or not. In fact, the joint information of all outcomes is actually useful and observation times alone are insufficient to infer the survival time distribution. Such approach is closely related to user's label privacy; see e.g. \cite{Cha11,Bei16,Wan19} where in this line of works, the authors consider differentially private mechanisms that allow the release of data without sensitive information. Under this setting, we propose an ad-hoc version of the conditional Nelson-Aalen estimator and show that it can achieve the minimax convergence rate when the bandwidth is correctly specified. 

\section{Model}

\subsection{Notations and estimators}

We consider a random vector $(T,C,X)$ taking values in $\R_+ \times \R_+ \times \R^p$ where $T$ defines the survival time, $C$ the censoring time and $X$ the covariate vector with support in $\mathcal{S}_X$ and density $f$. In the context of survival analysis, we assume that the observed data are restricted to $(Y,\delta,X)$ with support in $\mathcal{S}_Y\times \{0,1\}$ where $Y=T\wedge C$ is the observation time and $\delta=\ind_{\{T \leq C \}}$ is the failure or censoring indicator. Recall that under independent censoring, we assume that $T$ and $C$ are independent conditionally on $X$. The conditional distribution functions of the survival time and the censoring time are respectively denoted by $F_T$ and $F_C$ and given by 
\begin{equation*}
    F_T(t|x):= \P ( T \le t | X=x ) \quad \mathrm{and} \quad     F_C(t|x):= \P ( C \le t | X=x ),\quad\forall t\geq0
\end{equation*}
where $x\in\text{int}(\mathcal{S}_X)$ denotes a fixed covariate position. For any generic distribution function $F$, its queue function is denoted $\overline{F}=1-F$ for simplicity, which allows to write the conditional distribution function of $Y$ under independent censoring as
\begin{eqnarray*}
    H(t|x):=\P(Y\leq t|X=x)=1-\overline{F}_T(t|x)\overline{F}_C(t|x)).
\end{eqnarray*}
We will consider survival models with continuous random variables and as such denote $g$ for the conditional density function of $Y$ given $X$. We further denote by $\Lambda_T$ the conditional cumulative hazard function of the survival time $T$ with relationship between $\Lambda_T$ and $F$ given by
\begin{eqnarray*}
    1-F_T(t|x)=\exp(-\Lambda_T(t|x)).
\end{eqnarray*}
Let assume an independent and identically distributed (\textit{i.i.d.}) $n$-sized sample drawn from the censoring model $\{(Y_i,\delta_i,X_i)\}_{1\leq i\leq n}$. In our context, we also consider the privatized version $Z_i$ of the censoring indicators $\delta_i$ and recall that the sample $\{(Y_i,Z_i,X_i)\}_{1\leq i\leq n}$ is exclusively observed. The relationship between observation times and censoring indicators is characterized by the sub-distribution function $H^u$ which is defined by
\begin{eqnarray*}
        H^u(t|x):=\P(Y\leq t,\delta=1|X=x).
\end{eqnarray*}
Straightforward algebra shows that with a continuous random variable $T$ 
\begin{eqnarray*}
    \Lambda_T(t|x)=\int_0^t\dfrac{dH^u(s|x)}{1-H(s|x)}.
\end{eqnarray*}
This justifies the estimator construction of our privatized version $\widehat{\Lambda}_n$ given by
\begin{eqnarray*}
   \widehat{\Lambda}_n(t|x):=\int_0^t\dfrac{d\widehat{H}^u_n(s|x)}{1-H_n(s|x)}
\end{eqnarray*}
where $\widehat{H}^u_n$ and $H_n$ respectively denote empirical counterparts of the sub-distribution functions based on Nadaraya-Watson type estimators and given by
\begin{gather*}
H_n(t|x):=\sum_{i=1}^nW_h(x-X_i)\ind_{\{Y_i\leq t\}}\quad\text{and}\quad \widehat{H}_n^u(t|x):=\sum_{i=1}^nW_h(x-X_i)\ind_{\{Y_i\leq t\}}\widehat p_n(Y_i,X_i).
\end{gather*}
Note here that we consider weights
\begin{eqnarray*}
W_h(x-X_i):=\dfrac{K_h(x-X_i)}{\sum_{j=1}^nK_h(x-X_j)},\quad i=1,\ldots,n
\end{eqnarray*}
where $K_h(\cdot)=K(\cdot/h)/h^p$ with $K$ a kernel function and $h=h_n$ a non-random positive sequence such that $h_n\rightarrow 0$ as $n\rightarrow \infty$. The statistic $\widehat p_n$ denotes a privacy robust estimator of $p:=\mathbb{P}(T\leq C|Y=.,X=.)$ and has to be adequately chosen in order to replace the privatized censoring indicators. In this work, we propose another kernel type estimator given by
\begin{eqnarray}
\label{def::pn}
    \widehat p_n(y,x)=\sum_{j=1}^n\widetilde W_{b}((y,x)-(Y_j,X_j))Z_j,\quad y\in\mathbb{R}_+,\,x\in\mathcal{S}_X
\end{eqnarray}
where the weights are similarly defined as above with  
\begin{eqnarray*}
\widetilde W_{b}((y,x)-(Y_j,X_j)):=\dfrac{K_{b}(x-X_j)\widetilde K_b(y-Y_j)}{\sum_{k=1}^n K_b(x-X_k)\widetilde K_{b}(y-Y_k)},\quad j=1,\cdots,n
\end{eqnarray*}
and $\widetilde K$ and $b$ are equally defined but different kernel function and bandwidth than $K$ and $h$. Performance estimation will be addressed through the minimax framework by considering a pointwise risk as the mean integrated quadratic error of the privatized version of the Nelson-Aalen estimator with respect to the response variable $Y$. The criterion is then given by
\begin{equation*}
    \left\| \widehat{\Lambda}_n(.|x)- \Lambda_T(.|x) \right\|^2_{[t_0,t_1]}:=\E \left[\int_a^b \left(\widehat{\Lambda}(s|x) - \Lambda_T(s|x) \right)^2 ds\right]
\end{equation*}
where $[t_0,t_1]$ is any interval included in the interior of the support of $Y$ such that $H(t_1|x)<1$. This measure is rather classical in non-parametric functional analysis and makes particularly sense if one is interested in estimating the conditional distribution function of $T$ conditionally in $X=x$. Indeed, within the interval $[t_0,t_1]$, one can show that the discrepancy between $\Lambda_n$ and the Beran estimator \cite{Ber81} is relatively negligible (see for instance Lemma 2.1 in \cite{Esc23-3}) so that both estimators can be interchangeably used without lost of generality. In particular, straightforward application of the mean value theorem allows to have
\begin{eqnarray*}
    \left\| 1-\exp\left(-\widehat{\Lambda}_n(.|x)\right)- F_T(.|x) \right\|^2_{[t_0,t_1]}=O\left( \left\| \widehat{\Lambda}_n(.|x)- \Lambda_T(.|x) \right\|^2_{[t_0,t_1]}\right)
\end{eqnarray*}
so that  $1-\exp(-\hat{\Lambda}(.|x))$ consistently estimates the conditional distribution function $F_T$ with the same rate of convergence than that of $\hat{\Lambda}(.|x)$.\\

\subsection{Assumptions}
In order to prove the results in the sequel, we need to have some regularity assumptions on the model functions with Hölder type conditions. We here resume the required conditions:
\begin{itemize}
    \item $(\mathcal{H})$ : there exists $ 0 < \beta,\eta \le  1$ and a constant $c>0$ such that for any $t,s \in \mathcal{S}_Y$, $x,y \in \mathcal{S}_X$, 
    \begin{equation*}
    \begin{split}
        (\mathcal{H}.1) & \quad |f(x)-f(y)| \le c\|x-y\|^{\beta}, \\
        (\mathcal{H}.2) & \quad |g(t|x)-g(t|y)| \le c\|x-y\|^{\beta}, \\
        (\mathcal{H}.3) & \quad |H(t|x)-H(t|y)| \le c\|x-y\|^{\beta}, \\
        (\mathcal{H}.4) & \quad |F(t|x)-F(s|x)| \le c|t-s|^{\eta}, \\
        (\mathcal{H}.5) & \quad |H^u(t|x)-H^u(s|y)| \le c\left(|t-s|^{\eta}+\|x-y\|^{\beta}\right). \\
    \end{split}
    \end{equation*}
    We also denote $\mathcal{P}_{\eta,\beta}$ the space of the probability measures that defines the laws of the random vectors $(T,C,X)$ under assumption $\mathcal{H}$. Likely, we define the space of probability measures $\mathcal{P}^{(obs)}_{\eta,\beta}$ given by the laws of the random vectors $(Y,\delta,X)$ under assumption $\mathcal{H}$. Finally, we assume that the densities $f$ and $g$ are bounded uniformly on $\mathcal{S}_X$ and $\mathcal{S}_Y$.
    
    \item $(\mathcal{K})$ : the kernels $x\to K(x)$ and $(x,y)\to K(x)\widetilde K(y)$ are bounded density functions with supports the unit ball of $\R^p$ and $\R^{1+p}$ with respect to the euclidean norm $\|\cdot \|$. Likewise \cite{art:GineGuillou2002Rates}, we assume that the functions are square integrable and in the linear span (the set of finite linear combinations) of functions $k \geq 0$ satisfying the following property: the subgraph of $k$ ca be represented as a finite number of Boolean operations among sets of the form $\{(s,u): p(s,u)\geq\phi(u)\}$, where $p$ is a polynomial on either $\R^p\times\R$ or $\R^{1+p}\times \R$ and $\phi$ is an arbitrary real function. Alhtough this assumption seems quite technical, it is verified for a wide range of common kernel functions. We also assume that there exit positive constants $0<c_K,C_K$ (resp.  $0<c_{\widetilde{K}},C_{\widetilde{K}}$) such that $c_K\leq K(x)\leq C_K$ (resp. $c_{\widetilde{K}}\leq \widetilde K(x)\leq C_{\widetilde {K}}$) for all $x$ in the kernel support. 
\end{itemize}

\section{Local differential privacy}
\label{sec::ldp}


As mentioned in the introduction, we choose to preserve the privacy of the censoring indicators only. A particular reason is that we wish to maintain statistical efficiency of standard non-parametric survival statistics, like the Nelson-Aalen or Beran estimators. Another reason is that any random transformation of $Y$ might violate the independent censoring assumption, and thus potentially lead to biased estimates for most methods in the literature.\\
Let $\widetilde{\mathcal{Q}}_\alpha$ denotes the set of Markov kernels, also called channels, that take a random vector $(Y,\delta,X)$ as an input and publish $(Y,Z,X)$. For any $\widetilde{Q}\in\widetilde{\mathcal{Q}}_\alpha$, the resulting channel $Q$ that takes $\delta$ and gives $Z$ is $\alpha$-locally differentially private as defined in (\ref{eq::ldp}). We here consider that a censoring indicator is privatized by the addition of an independent random variable with centred Laplace distribution. Let $f^{(obs)}$ be the density of $(Y,\delta,X)$, it follows that $(Y,Z,X)$ has a density $m$ against the Lebesgue measure on $\R^2$ and the law of $X$ such that
\begin{eqnarray*}
    m(t,z,x)&:=&\sum_{b=0}^1  q_\alpha(z|b) f^{(obs)}(t,b,x)\\
    &=&\sum_{b=0}^1  q_\alpha(z|b)\left[\,\overline{F}_C(t|x) f_T(t|x)\right]^{b} \left[\,\overline{F}_T(t|x) f_C(t|x) \right]^{1-b} 
\end{eqnarray*}
where the density channel function is given by
\begin{eqnarray*}
    q_\alpha(z|b):=\dfrac{\alpha}{2}\exp\left(-\alpha|z-b|\right),\quad\forall z\in\R.
\end{eqnarray*}
As expected, the privacy mechanism is $\alpha$-differentially private in a non-interactive way and belongs to the expected set of channels.
\begin{lem}
\label{lem::privacy}
    Let $\widetilde{Q}$ defines the privacy mechanism that takes $(Y,\delta,X)$ and returns $(Y,Z,X)$ as described above. Then $\widetilde{Q}\in\widetilde{\mathcal{Q}}_\alpha$.
\end{lem}
\noindent
Label differential privacy proposes an alternative definition of privacy constraint with scenarios where some of the outcomes are not necessarily sensitive and can be made publicly available. This definition was introduced in \cite{Cha11}, generalized in \cite{Bus23} and based on the Rényi divergence.
\begin{definition}
    For any $\gamma>1$, the Rényi $\gamma$-divergence between distributions $P$ and $Q$ is given by
    \begin{eqnarray*}
        D_\gamma(P||Q)=\dfrac{1}{\gamma-1}\log\E_Q\left[\left(\dfrac{dP}{dQ}(z)\right)^\gamma\right]
    \end{eqnarray*}
    if the Radon-Nykodim density $\frac{dP}{dQ}$ is well defined and $+\infty$ otherwise.
\end{definition}
\noindent
It follows then that for measurable spaces $(S_i,\mathcal{S}_i)$, $i=1,2$, any channel $Q:(S_1\times S_2)\times(\mathcal{S}_1\otimes\mathcal{S}_2)\to[0,1]$ is $(\gamma,\varepsilon)$-label Rényi differentially private if for all $y,y'\in S_2$ and $x\in S_1$
\begin{eqnarray*}
    D_\gamma(Q(x,y)||Q(x,y'))\leq\varepsilon.
\end{eqnarray*}
where we allow channels to reveal parts of the inputs. In the following result, we show that our privacy channel $\widetilde{Q}$ also preserves user's label privacy with parameters depending from the differential privacy index $\alpha$.
\begin{lem}
\label{lem::labelprivacy}
    Let $\widetilde{Q}$ defines the privacy mechanism that takes $(Y,\delta,X)$ and returns $(Y,Z,X)$ as described above. Then $\widetilde{Q}$ is $(\gamma,\varepsilon)$-label Rényi differentially private for any $\gamma>1$ and $\varepsilon=\alpha$.
\end{lem}
\noindent

\section{Risk optimization}
In this section, we derive lower bounds for the minimax risk of the cumulative hazard function estimation and provide risk upper bounds for $\widehat{\Lambda}_n$. In particular, we show that the minimax rate of convergence is reached for our estimator when the bandwidth $h$ is selected according to the parameter $\beta$.\\

\textbf{Minimax risk}: we consider lower risk bounds over the set of distributions $\mathcal{P}_{\eta,\beta}$ and the $\alpha$-locally differentially private mechanisms $\widetilde{Q}$, as described in Section \ref{sec::ldp}. Note that we distinguish the set $\mathcal{P}_{\eta,\beta} $ of distributions $(T,C,X)$ which define the latent model and the set $\mathcal{P}_{\eta,\beta}^{(obs)} $ of the observation distributions $(Y,\delta,X)$. Concentration of the covariate distribution around $x$ is controlled throughout small ball probabilities; see e.g. \cite{art:ChagnyRoche_AdaptiveMinimax}, \cite{ferraty2006nonparametric}. We here consider the simple case where there exist some constants $c_X$, $C_X>0$, $\gamma >0$ such that for any $h>0$
\begin{equation*}
    c_X h^\gamma \le \P ( \| X-x \| \le h ) \le C_X h^\gamma.\tag{$\mathcal{H}_X$}
\end{equation*}
This assumption is quite reasonable for distributions with support on bounded subsets and  is less restrictive than the model assumptions in $\mathcal{H}$. In our context, we will show in Lemma \ref{lem:smallballHolderversion} that this assumption is fulfilled with $\gamma=p$. We are now ready to derive the convergence rates for the minimax risk.
\begin{theorem}\label{theorem:minimax} Assume $(\mathcal{H})$ and $(\mathcal{H}_X)$. Let $x \in \text{int}(\mathcal{S}_X)$ and $\alpha \in (0,1)$, then we have 
\begin{equation*}
     \inf_{\tilde Q \in \mathcal{\tilde Q_\alpha}}\inf_{ \widetilde{\Lambda}_n} \sup_{P \in \mathcal{P}_{\eta,\beta}} \E_{P,\tilde Q} \left[  \left\| \widetilde{\Lambda}_n(.|x)- \Lambda_{T(P)}(.|x)\right\|^2_{[t_0,t_1]} \right] \ge C_{t_0,t_1}\left(\frac{1}{\alpha^2n }\right)^{\frac{2\beta}{2\beta+ \gamma}}
\end{equation*}
where $T(P)$ denotes the random survival time issued from the distribution $P$, the infimum is over all the possible estimator $\widetilde{\Lambda}_n$ of the sample $\{(Y_i,Z_i,X_i),i=1,\ldots,n\}$ and $C_{t_0,t_1}$ is an explicit constant given in the proof.
\end{theorem}
\noindent
Note that in this theorem, we have only studied the minimax rate for a pointwise risk in $x$, although a similar result can be shown for an integrated risk following the arguments of \cite{book:Tsybakov_nonparam} and \cite{art:ChagnyRoche_AdaptiveMinimax}. We however have chosen to omit the proof since the upper bound developed in the following part applies only for the pointwise risk. It also is noteworthy that the rate of convergence almost meets the same optimal rate obtained by \cite{art:ChagnyRoche_AdaptiveMinimax} without any privacy mechanism. In particular, we see that when the privacy level $\alpha$ increases, it decreases the optimal statistical efficiency that one can expect from the best estimator applied to the randomized data.\\
\bigskip

\textbf{Estimator's risk}: in order to obtain risk upper bounds for the estimator $\widehat{\Lambda}_n$, we consider Proposition 2.1 in \cite{Esc23-3} about the almost-sure representation of the Nelson-Aalen estimator with generic indicators. The result is applied in our context with indicators replaced by the privatized version $Z_i$'s. The independence between the privacy mechanism and the prior survival model shows that the privatized indicators return the same conditional expectation than that of the censoring indicators, that is
\begin{eqnarray*}
    \E\left[Z|Y,X\right]=\P(\delta=1|Y,X),\quad a.s
\end{eqnarray*}
which already ensures that our estimator $\widehat{\Lambda}_n$ is unbiased according to Corollary 2.1 in \cite{Esc23-3}. 

\begin{theorem}\label{Theo:Upperbound} Assume $(\mathcal{K})$ and $(\mathcal{H})$ with $nh^{2\beta+p}|\log h|^{-1}=\mathcal{O}(1)$ and $h=\mathcal{O}(b^{(1+p)/p})$. Let $x \in \text{int}(\mathcal{S}_X)$ such that $\inf_{y\leq t_1}g(y|x)>0$, then there exists a constant $D_{t_0,t_1}$ with
    \begin{eqnarray*}
        \E_{P,\widetilde{Q}} \left[  \left\| \widehat{\Lambda}_n(.|x)- \Lambda_{T(P)}(.|x) \right\|^2_{[t_0,t_1]} \right] \le D_{t_0,t_1} \left(h^{2\beta}  + \frac{1}{n  h^p}\left(1 +\frac{1}{\alpha^2}\right) \right)
    \end{eqnarray*}
    such that $\widetilde{Q}$ is the privacy mechanism described above and $P\in\mathcal{P}_{\eta,\beta}$. Note that we retrieve the same convergence rate than that of Theorem \ref{theorem:minimax} with $h=[\alpha^2 n]^{-\frac{1}{2\beta+p}}$.
\end{theorem}
\noindent
We observe that additional convergence rates appears in the upper bounds. The reason is that the almost-sure representation of the estimator takes the form of a sum of different statistics where not all of them are impacted by the privacy mechanism. On one hand, the rate $h^{2\beta}$ results from the estimator bias and the regularity conditions of the model functions, which partially appears in the risk upper bounds in \cite{art:ChagnyRoche_AdaptiveMinimax}. On the other hand, the rate $1/nh^p$ results from the remaining statistics where no private outcomes are considered and actually represent the rate of convergence for the Nelson-Aalen estimator without privacy. The rate $h^{2\beta}$ is the dominant convergence term which guarantee the estimator to be optimal if the the bandwidth is tuned accordingly to the model functions' regularity.  

\section{Simulations}

In this section, we propose a short simulation studies in order to illustrate the behaviour of our privacy procedure under different settings, all conducted with the language \texttt{R}. The random variable $X$ is distributed according to an uniform law on $[0,1]$ and both the random times $T$ and $C$ are built from $X$ according to exponential laws respectively with rates $\langle \lambda_T,(1,X,X^2)\rangle$ and  $\langle \lambda_C,(1,X,X^2)\rangle$ where $\lambda_T,\lambda_C\in\R^3$. Note that the choice of vectors $\lambda_T$ and $\lambda_C$ allows us to control the level of censorship in our model given by
\begin{equation*}
    \P(T \le C |X=x) = \frac{\langle \lambda_T, (1,x,x^2) \rangle }{\langle \lambda_T+\lambda_C, (1,x,x^2) \rangle},\quad x\in[0,1].
\end{equation*}
We construct the privatized data using \textit{i.i.d.} Laplace random variables $\{\mu_k\}_{1\leq k\leq n}$ with parameter and privacy level $\alpha >0$. The privatized observed vector is then given by $(Y_k,Z_k,X_k)_{1\leq k\leq n}$ where $Z_k=\delta_k + \mu_k $. The efficiency degradation due to our privacy procedure will be discussed with comparison  between our method and the Beran estimator $F_n$ without privacy. We furthermore denote the estimator with cleaned data $\widecheck{F}_n$, also referred to as generalized Beran estimator (see \cite{Esc23-3} for more details) so that the conditional probability estimator is given by 
\begin{eqnarray}
\label{def::pnnonprivate}
    \widecheck p_n(y,x)\coloneqq\sum_{j=1}^n\widetilde W_{b}((y,x)-(Y_j,X_j))\delta_j,\quad y\in\mathbb{R}_+,\,x\in\mathcal{S}_X.
\end{eqnarray}
Overall, we consider the uniform kernel function $K:x\in\mathbb{R}\to\frac{1}{2}\ind_{\{|x|\le 1\}}$ and select the bandwidth $h$ as 5 times the return of the function \texttt{dpik} from the \texttt{R}-package \texttt{KernSmooth}. The conditional probability estimators in \eqref{def::pn} and \eqref{def::pnnonprivate} are computed with $\widetilde{K}=K$ and $b=\sqrt{h}$, in accordance with the assumptions of Theorem \ref{Theo:Upperbound}. Several series of simulations are done for each scenarios with $N=300$ samples of size $n=500$. As a comparative measure, we consider the mean squared error (MSE) 
\begin{eqnarray*}
    \mathrm{MSE}_x(t)= \frac{1}{N}\sum_{k=1}^N \left( E^{(k)}(t|x)- F_T(t|x) \right)^2,\quad t\geq0,\,x\in[0,1]
\end{eqnarray*}
where $E^{(k)}$ is any estimators $\widehat{F}_n$, $\widecheck{F}_n$ or $F_n$ based on the $k$-th sample. We performed several simulations with various values of $x$ and observed no significant differences between the performances. We thus have chosen to only display here the results for $x=0.5$. Two censoring proportions are displayed with $25\%$ and $50\%$ respectively based on the vector combinations $(1,1,1)=\lambda_T=\lambda_C$ and $(1,1,1)=\lambda_T=\lambda_C/3$. We make vary the privacy level from more to less privacy by selecting $\alpha=0.2,0.3$ and $0.4$. The simulation results of our experiment are regrouped in Figure \ref{fig:CompareMSEscenario1} with the MSE curves in function of time $t$.\\
Without privacy, the generalized Beran estimator shows the smallest variability and always outperforms its regular counterpart. This has already been discussed and studied in \cite{Esc23-3} and will serve as standard performances to evaluate the statistical efficiency of our privacy procedure. Overall, we observe that our approach returns the highest MSE values when $\alpha$ is small, while relaxing the privacy constrain with larger $\alpha$ values returns MSE curves similar to the generalized Beran estimator. This is expected since the privacy procedure is based on noise perturbed indicators, which drastically interfere in the estimator stability. In a sense, the performances of the generalized Beran estimator represents the best attainable results since both statistics share the same definition when $\alpha\to +\infty$. It particularly appears that our methods shares similar curves than the Beran estimator without privacy when $\alpha=0.3$. Nevertheless, the simulations have shown that biases among the three different methods are similar, meaning that our estimator remains as efficient as non-private procedures. Censoring proportion also interfere in the estimation efficiency with higher MSE curves when censoring increases. However, this behaviour equally affects all the methods and does not change the latter comparison. By averaging the privatized indicators, $\widehat{p}_n$ is robust despite the noise perturbation and provide reasonable probability estimates when used in the estimation of the survival distribution. This conveniently ensures that we can obtain a private procedure with equalled performances than that of standard non-private methods.


\begin{figure}[H]
    \centering
    \includegraphics[scale=.45,trim= 4cm 0cm 1.5cm 0cm]{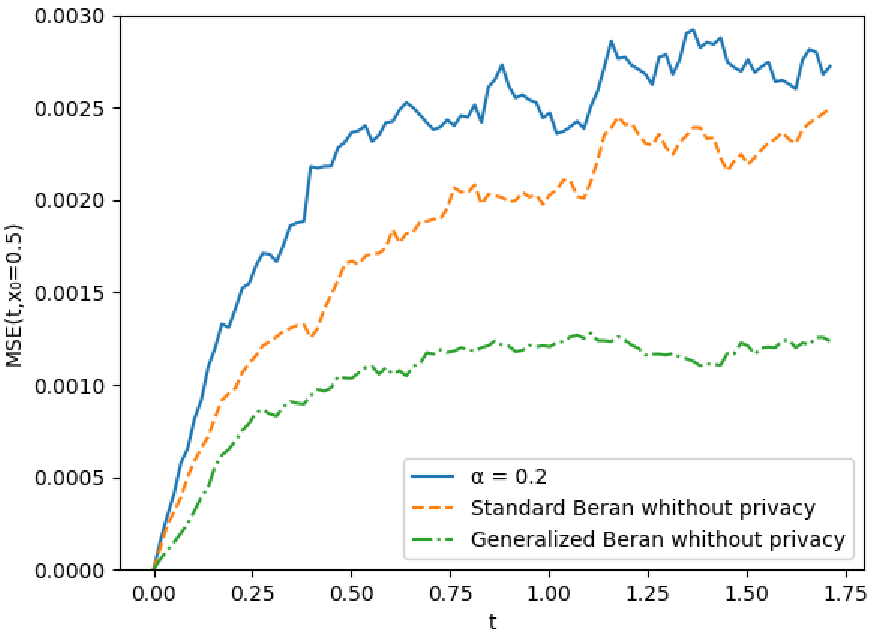}
    \includegraphics[scale=.45,trim= 2.05cm 0cm 1.5cm 0cm,clip]{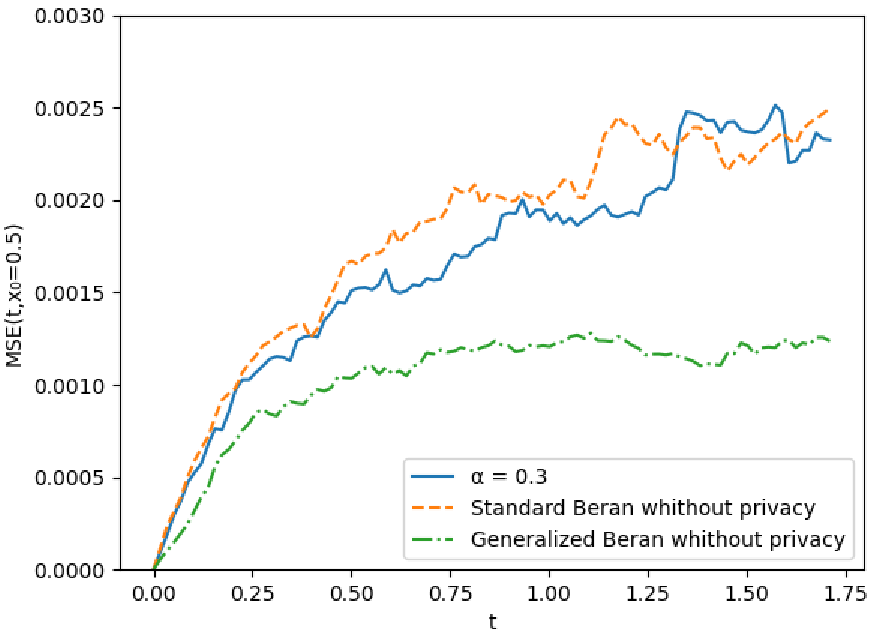}
    \includegraphics[scale=.45,trim= 2.05cm 0cm 1.5cm 0cm,clip]{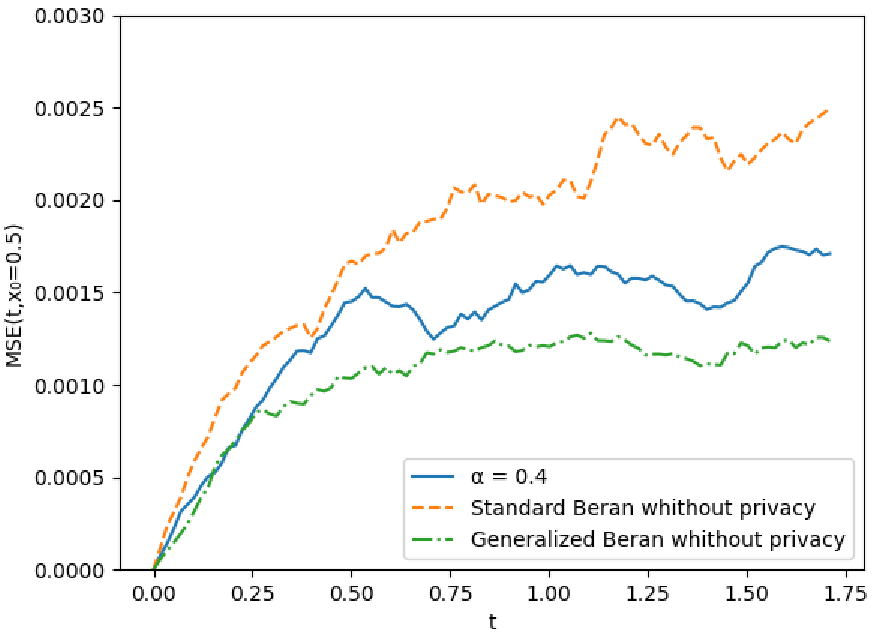}
    \includegraphics[scale=.45,trim= 4cm 0cm 1.5cm 0cm]{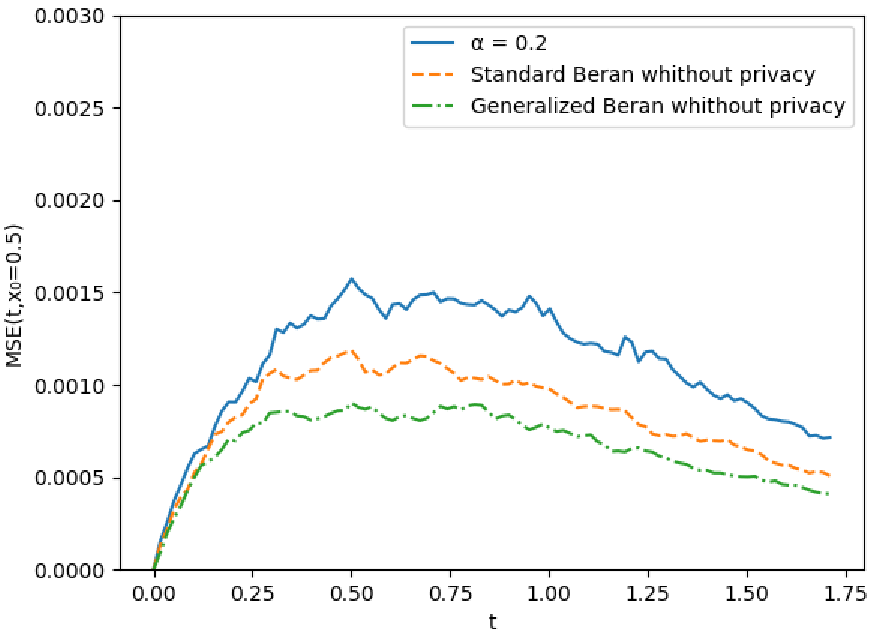}
    \includegraphics[scale=.45,trim= 2.05cm 0cm 1.5cm 0cm,clip]{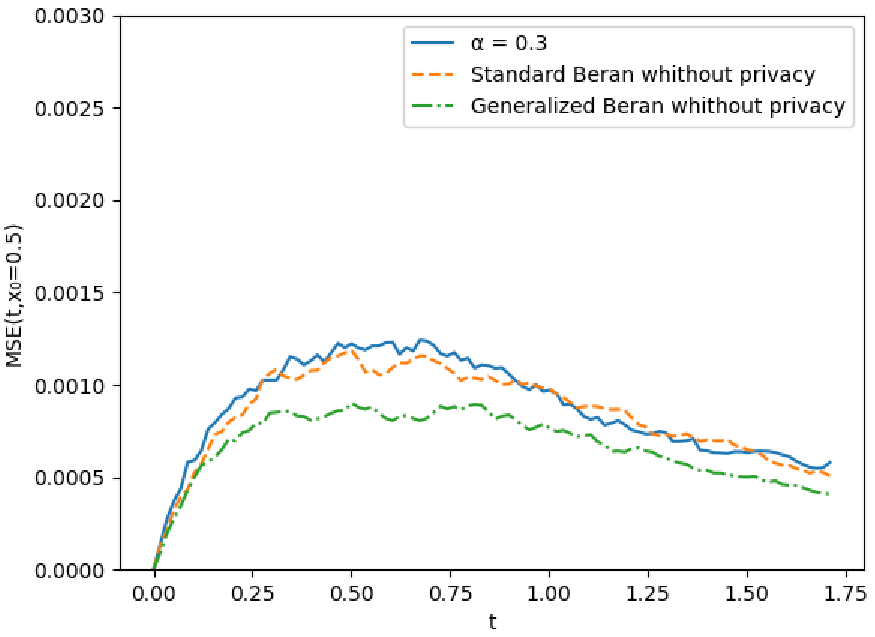}
    \includegraphics[scale=.45,trim= 2.05cm 0cm 1.5cm 0cm,clip]{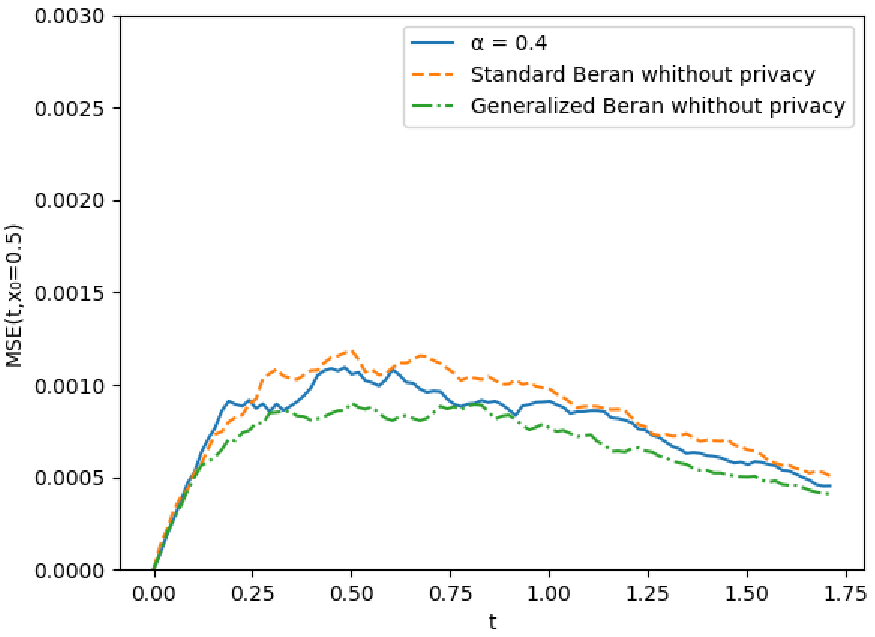}
    \caption{Graph of Mean Squared Error (MSE) as a function of time with $50 \%$ of censoring for the first line and $25\%$ for the second.}
    \label{fig:CompareMSEscenario1}
\end{figure}

\section{Proof}  \label{sec:proof}

\noindent
\textbf{Proof of Lemma \ref{lem::privacy}}: We use an equivalence between the definition in (\ref{eq::ldp}) and the channel density functions. Indeed, we have that a mechanism is $\alpha$-locally differentially private if and only if we have 
\begin{eqnarray*}
    \sup_{z\in\R}\sup_{x,x'}\dfrac{q(z|x)}{q(z|x')}\leq e^\alpha
\end{eqnarray*}
where $q$ is the channel conditional density function. In our context, we have for any $z\in\R$ and $b,b'\in \{0,1\}$
\begin{eqnarray*}
    \dfrac{q_\alpha(z|b)}{q_\alpha(z|b')}&=&\exp\left(\alpha|z-b|-|z-b'|\right)\leq\exp\left(\alpha|b-b'|\right)\leq e^\alpha
\end{eqnarray*}
and the result follows.\hfill\qed\\

\noindent
\textbf{Proof of Lemma \ref{lem::labelprivacy}}: the proof is similar to the previous one and makes use of the ratio between the kernel density functions. By definition, for any $t\geq0$, $z\in\R$ and $b\in\{0,1\}$, the conditional density of the random vector $(Y,Z,X)$ given $(Y,\delta,X)=(y',b,x')$ against the measure $\lambda(dz)\otimes\P(Y\in ., X\in .)$ is given by 
\begin{eqnarray*}
    m(y,z,x|y',b,x')=q_\alpha(z|b)\ind_{\{y=y',x=x'\}}
\end{eqnarray*}
for any $x\in\mathcal{S}_X,\,y\geq 0\text{ and }z\in\R$. This allows us to have that
\begin{eqnarray*}
    \dfrac{m(y,z,x|y,b,x)}{m(y,z,x|y,b',x)} = \dfrac{q_\alpha(z|b)}{q_\alpha(z|b')}\leq e^\alpha
\end{eqnarray*}
and the result follows. 

\subsection{Technical lemmas}
Before coming into the proof of the main results, we preface some technical lemmas useful in the derivation of deterministic upper bounds. Lemma \ref{lem:smallballHolderversion} provides concentration results for small ball probabilities with covariates satisfying the Hölder assumption $(\mathcal{H})$ and Lemma \ref{lem:BoundonEspK} proves similar results with kernel expectations. Based on the proofs in \cite{ferraty2006nonparametric} and \cite{art:ChagnyRoche_AdaptiveMinimax}, we will consider deviations given in Lemma \ref{lem:deviationRhx} for the following processes
\begin{eqnarray*}\label{eq:processRH}
    &&R_h^x = \frac{1}{n} \sum_{i=1}^n \frac{K_h(x-X_i)}{\E \left[ K_h(x-X)\right]}\quad\text{and}\quad R_{b}^{y,x} = \frac{1}{n} \sum_{i=1}^n \frac{K_b(x-X_i)\widetilde K_b(y-Y_i)}{\E \left[K_b(x-X)\widetilde K_b(y-Y)\right]}
\end{eqnarray*}
where $h,b\in(0,1]$. Finally, bias of the statistics appearing in the main result proof will be controlled with bounds on the expectations $\E \left[  K_h ( x -X)\ind_{\{Y \le t \}}\right]$ and $\E \left[  K_h ( x -X)\ind_{\{Y \le t,\delta=1 \}}\right]$ proposed in Lemma \ref{lem:errkernelestimH}.

\begin{lem}\label{lem:smallballHolderversion}
Assume $(\mathcal{H})$. Let $h\in (0,1)$, $y\in\text{int}(\mathcal{S}_Y)$ and $x\in \text{int}(\mathcal{S}_X)$. Then, there exist measurable functions $x\to\psi^x_n$ and $(x,y)\to \phi^{x,y}_n$ such that
\begin{eqnarray*}
        \P( \|x-X\|\le h)=h^p(f(x)+\psi^x_n)
\end{eqnarray*}
\begin{center}
    and
\end{center}
\begin{eqnarray*}
    \P( \|x-X\|\le b,|y-Y|\leq b) = b^{1+p}(g(y|x)f(x)+\phi^{x,y}_n)
\end{eqnarray*}
where $\sup_x|\psi^x_n|=\mathcal{O}(h^\beta)$ and $\sup_{x,y}|\phi^{x,y}_n|=\mathcal{O}(b^\eta+b^\beta)$.
\end{lem}
\begin{proof}
By a change of variable we have
\begin{equation*}
    \begin{split}
        \P( \|x-X\|\le h)   & = \int \ind_{\{\|x-y \| \le h \}}f(y) \mathrm{d}y 
        \\ & = h^p \int \ind_{\{\|u \| \le 1 \}}f(x-hu) \mathrm{d}u 
        \\ & = h^p \int_{B(0,1)} f(x-hu)-f(x) \mathrm{d}u + h^p  f(x) |B(0,1)|.
    \end{split}
\end{equation*}
But, under Assumption $(\mathcal{H})$ we have
\begin{equation*}
    -h^{p+\beta}c \int_{B(0,1)} \|u\|^\beta \mathrm{d}u \le h^p \int_{B(0,1)}  f(x-hu)-f(x) \mathrm{d}u \le h^{p+\beta}c \int_{B(0,1)} \|u\|^\beta \mathrm{d}u
\end{equation*}
we then obtain
    \begin{equation*}
        \left| \P\left( \|x-X\|\le h\right) -  h^pf(x) |B(0,1)| \right|\le h^{p+\beta}  c \int_{B(0,1)}  \|u\|^\beta \mathrm{d}u
    \end{equation*}
which leads to the first result. We conduct a similar analysis for the second expression based on the Hölder assumptions for the model functions $f$ and $g$. Specifically, by use of a change of variables
\begin{eqnarray*}
    \P( \|x-X\|\le b,\,|y-Y|\leq b)&=&\int f(u)\ind_{\{\|x-u\|\leq b\}}\int g(v|u)\ind_{\{|y-v|\leq b\}}dvdu\\
    &=:&\int f(u)\ind_{\{\|x-u\|\leq b\}}G_n(u)dvdu
\end{eqnarray*}
where we observe that for any $u\in\mathcal{S}_X$
\begin{eqnarray*}
   \left|\int g(v|u)\ind_{\{|y-v|\leq b\}}dv - bg(y|u) \right|\leq cb^{1+\eta}\int_{-1}^1|s|^\eta ds.
\end{eqnarray*}
We obtain that $G_n=b(g(y|\cdot)+\mathcal{O}(b^\eta))$ uniformly, which implies that
\begin{eqnarray*}
    \int f(u)\ind_{\{\|x-u\|\leq b\}}G_n(u)du&=&\int f(u)\ind_{\{\|x-u\|\leq b}b(g(y|u)+\mathcal{O}(b^\eta))du\\
    &=&b^{1+p}\left(\int_{B(0,1)} f(x-ub)g(y|x-ub)du+\mathcal{O}(b^{\eta})\right)\\
    &=&b^{1+p}\left(f(x)g(y|x)+\mathcal{O}(b^{\beta})+\mathcal{O}(b^{\eta})\right)
\end{eqnarray*}
uniformly in $u\in\mathcal{S}_X$., which ensures that
\begin{eqnarray*}
    \P(\|x-X\|\leq b,|y-Y|\leq b)=b^{1+p}\left(f(x)g(y|x)+\mathcal{O}(b^{\beta})+\mathcal{O}(b^{\eta})\right).
\end{eqnarray*}
which concludes the proof.
\end{proof}
\bigskip

\begin{lem}\label{lem:BoundonEspK} Assume $(\mathcal{K})$ and $(\mathcal{H})$. Let $h\in (0,1)$, $y\in\text{int}(\mathcal{S}_Y)$ and $x \in \text{int}(\mathcal{S}_X)$, then for any $l >0$
    \begin{eqnarray*}
    [c_Kh^{-p}]^l  \le&\dfrac{\E \left[K_h\left(x-X \right)^l \right]}{\P \left( \| x-X \| \le h \right)}& \le [C_K h^{-p}]^l 
\end{eqnarray*}
\begin{center}
    and
\end{center}
\begin{eqnarray*}
    \left[c_Kc_{\widetilde K}b^{-1-p}\right]^l  \le&\dfrac{\E \left[(K_b(x-X)\widetilde K_b(y-Y))^l\right]}{\P( \|x-X\|\le b,|y-Y|\leq b)}& \le \left[C_KC_{\widetilde K}b^{-1-p}\right]^l
\end{eqnarray*}
\end{lem}
\begin{proof}
    By definition of $K_h$
        \begin{equation*}
            \E \left[K_h\left(x-X \right)^l \right]  = h^{-pl} \E\left[ K\left(\frac{x-X}{h} \right)^l \right]. 
        \end{equation*}
    The function $K$ is assumed compactly supported on the unit ball, which gives us
        \begin{equation*}
            \E \left[K_h\left(x-X \right)^l \right]  = h^{-pl} \E\left[ \ind_{\{\|x-X\|\le h\}}K\left(\frac{x-X}{h} \right)^l \right].
        \end{equation*}
    and the first assertion follows since $K$ is bounded. Similar arguments allow to show the second and third expressions.
\end{proof}

\bigskip
\begin{lem}\label{lem:deviationRhx}Assume $(\mathcal{K})$ and $(\mathcal{H})$. Let $h\in (0,1)$ $y\in\text{int}(\mathcal{S}_Y)$ and $x \in \text{int}(\mathcal{S}_X)$, then the following inequalities holds
    \begin{equation*}
         \P\left(R^x_h \le 1/2 \right) \le  2 \exp\left(-c_1nh^p(f(x)+\psi^x_n)\right)\quad\text{and}\quad \P\left(R^{y,x}_b \le 1/2 \right) \le  2 \exp\left(- c_2 nb^{1+p}(f(x)g(y|x)+\phi^{x,y}_n)\right)
    \end{equation*}
where $\psi_n^x$ and $\phi_n^{x,y}$ are defined in Lemma \ref{lem:smallballHolderversion} and
\begin{eqnarray*}
   c_1:=\left[8\left(\frac{C_K^2}{c_K^2}+\frac{C_K}{2c_K}\right)\right]^{-1} \quad\text{and}\quad c_2:=\left[8\left(\frac{C_K^2C^2_{\widetilde K}}{c_K^2c^2_{\widetilde K}}+\frac{C_KC_{\widetilde K}}{2c_Kc_{\widetilde K}}\right)\right]^{-1}.
\end{eqnarray*}
\end{lem}
\begin{proof}
     Based on the event inclusion of $\{R^x_h \le 1/2 \}$ into $\{|R^x_h - 1 | \ge 1/2 \}$, we derive an upper bound for $\P\left(|R^x_h - 1 | \ge 1/2  \right) $. This deviation term will be controlled thanks to Bernstein's inequality (see \cite[Lemma 8]{art:birgeMassart_minimumconstrastestimator}) which we state as followed. Let $V_1, \cdots, V_n$ be independent random variables and $S_n(V) = \sum_{i=1}^n (V_i - \E [V_i])$. Then for any $b$ and $v$ be positive constants with
     \begin{equation*}
         \forall l \ge 2, \quad \frac{1}{n}\sum_{i=1}^n \E [|V_i|^l] \le \frac{l!}{2}v^2 b^{l-2}
     \end{equation*}
     we have for $\eta>0$
     \begin{eqnarray}
     \label{eq:bernstein}
         \P \left(\frac{1}{n}|S_n(V)| \ge \eta \right) \le 2 \exp\left(-\frac{n\eta^2}{2(v^2+b\eta)}\right).
     \end{eqnarray}
     We use this lemma in our context by considering the random variables 
     \begin{eqnarray*}
        V_i:=\frac{K_h(x-X_i)}{\E[K_h(x-X)]},\quad i=1,\cdots,n 
     \end{eqnarray*}
     with $S_n(V)/n=R_h^x -1$. By Lemma \ref{lem:BoundonEspK}, we have for $l\ge2$
    \begin{equation*}
         \E \left[|V_i|^l \right]= \frac{ \E \left[K_h^l ( x - X)\right]}{\E \left[K_h(x-X)\right]^l}\le \frac{C_K^l}{c_K^l \P( \|x-X\|\le h)^{l-1}},
     \end{equation*}
     which implies that
    \begin{equation*}
         \frac{1}{n}\sum_{i=1}^n \E [|V_i|^l] \le  \underbrace{\frac{C_K^2}{c_K^2 \P( \|x-X\|\le h)}}_{\eqqcolon v^2}  \underbrace{\left( \frac{C_K}{c_K \P( \|x-X\|\le h)}\right)^{l-2}}_{\eqqcolon b^{l-2}},
     \end{equation*}
     The first results finally follows from (\ref{eq:bernstein}) with $\eta=1/2$ and Lemma \ref{lem:smallballHolderversion}. The same arguments can be applied to the random variables
     \begin{eqnarray*}
        \frac{K_b(x-X_i)\widetilde K_b(y-Y_i)}{\E[K_b(x-X)\widetilde K_b(y-Y)]},\quad i=1,\cdots,n
     \end{eqnarray*}
     so that we obtain similar concentration results for the process $R_b^{x,y}$ which concludes the proof of the lemma.
\end{proof}
\bigskip

\begin{lem}\label{lem:errkernelestimH} Assume $(\mathcal{K})$ and $(\mathcal{H})$ . Then for any $h \in (0,1)$,
    \begin{equation*}\label{eq:forH}
        \E \left[  K_h ( x -X)\ind_{\{Y \le t \}}\right] -  H(t|x) \E \left[K_h(x-X)\right] \le  h^{\beta} C_{K,X,\beta},
    \end{equation*}
    and
    \begin{equation*}\label{eq:forHu}
        \E \left[  K_h ( x -X)\ind_{\{Y \le t, \delta=1 \}}\right] -  H^u(t|x) \E \left[K_h(x-X)\right] \le  h^{\beta} C_{K,X,\beta},
    \end{equation*}
    where $C_{K,X,\beta}$ is given in \eqref{eq:constantforHandHu}.
\end{lem}
\begin{proof}
    The proof for $H^u$ and $H$ are exactly the same and we thus only give the proof for the latter. Since $X$ has a density $f$, we have 
    \begin{equation*}
        \begin{split}
            \E \left[  K_h ( x -X)\ind_{\{Y \le t \}}\right] &= \int_{\R^p} K_h(x-y) \E \left[ \ind_{\{Y \le t \}} |X=y \right]f(y)\mathrm{d}y
            \\ & = \int_{\R^p} K_h(x-y) H(t|y) f(y)\mathrm{d}y.
        \end{split}
    \end{equation*}
    By definition of $K_h$ and a change of variable, we have
    \begin{equation*}
        \begin{split}
            \E \left[  K_h ( x -X)\ind_{\{Y \le t \}}\right] & = \int_{\R^p}h^{-p} K\left(\frac{x-y}{h}\right)H(t|y)  f(y)\mathrm{d}y
            \\ & = \int_{\R^p} K(y)H(t|x-hy)  f(x-hy)\mathrm{d}y,
        \end{split}
    \end{equation*}
    and since $H$ is in the Hölder class, we have 
        \begin{equation*}
        \begin{split}
            \E \left[  K_h ( x -X)\ind_{\{Y \le t \}}\right] &= \int_{\R^p} K(y)H(t|x-hy)  f(x-hy)\mathrm{d}y -  H(t|x) \E \left[K_h(x-X)\right]+  H(t|x) \E \left[K_h(x-X)\right]
            \\ & \le \int_{\R^p} K(y)\left|H(t|x-hy)-H(t|x)\right| f(x-hy)\mathrm{d}y +  H^u(t|x) \E \left[K_h(x-X)\right]
            \\ & \le h^{\beta} c \int_{\R^p} K(y)\left\|y\right\|^\beta f(x-hy)\mathrm{d}y +  H(t|x) \E \left[K_h(x-X)\right],
        \end{split}
    \end{equation*}
    Finally, with the Hölder property of $f$, we have 
    \begin{equation} \label{eq:constantforHandHu}
        \begin{split}
            \E \left[  K_h ( x -X)\ind_{\{Y \le t \}}\right] & -  H(t|x) \E \left[K_h(x-X)\right] 
            \\ & \le h^{\beta} c \int_{\R^p} K(y)\left\|y\right\|^\beta f(x-hy)-f(x)\mathrm{d}y + h^{\beta} c \int_{\R^p} K(y)\left\|y\right\|^\beta f(x)\mathrm{d}y 
            \\ & \le h^{2\beta} c^2 \int_{\R^p} K(y)\left\|y\right\|^{2\beta} \mathrm{d}y + h^{\beta} c \int_{\R^p} K(y)\left\|y\right\|^\beta f(x)\mathrm{d}y 
            \\ & \le h^{\beta} \underbrace{\left(h^{\beta}  c^2 \int_{\R^p} K(y)\left\|y\right\|^{2\beta} \mathrm{d}y +  c \int_{\R^p} K(y)\left\|y\right\|^\beta f(x)\mathrm{d}y \right)}_{C_{K,X,\beta}},
        \end{split}
    \end{equation}
    which concludes the proof.
\end{proof}

\subsection{Proof of the Theorem \ref{Theo:Upperbound}}
Let $p(y,x):=\mathbb{P}(\delta=1|Y=y,X=x)$ defines the conditional success probability of $\delta$ given the survival outcomes. We consider the generalized Nelson-Aalen estimator
\begin{eqnarray*}
    \Lambda_n:=\int_0^t\dfrac{dH^u_n(s|x)}{H_n(s|x)}
\end{eqnarray*}
where $H^u_n$ denotes the kernel type estimator of $H^u$ based on the indicators $\{p(Y_i,X_i)\}_{1\leq i\leq n}$ and given by
\begin{eqnarray*}
    H_n^u(t|x):=W_h(x-X_i)\ind_{\{Y_i\leq t\}}p(Y_i,X_i).
\end{eqnarray*}
Our proof is based on the decomposition $\widehat\Lambda_n-\Lambda=\widehat\Lambda_n-\Lambda_n+\Lambda_n-\Lambda$ where we will separately consider the privatized and cleared terms. The aim of this strategy is to simplify the analysis by isolating the terms dependent from the privacy, where we see that only $\widehat p_n$ will require adapted computations. In the first part, results in \cite{art:GineGuillou2002Rates} based on Talagrand's inequality and Vapnik–Chervonenkis (VC) classes will be used to derive concentration results for $\Lambda_n-\Lambda$, and in the second part, strategies similar to the proofs in \cite{art:ChagnyRoche_AdaptiveMinimax} will be adapated.

\subsubsection{Risk analysis for $\Lambda_n-\Lambda$}
Application of Proposition 2.1 in \cite{Esc23-3} ensures that under $(\mathcal{H})$ and $(\mathcal{K})$, we have
\begin{eqnarray*}
    \Lambda_n(t|x) - \Lambda(t|x) = \sum_{i=1}^n W_h (x - X_i) \ell (t,Y_i,\delta_i,Z_i|x )) + r_n(t|x),
\end{eqnarray*}
where 
\begin{equation*}
    \begin{split}
         \ell (\cdot,Y_i,\delta_i,Z_i|x) & = \underbrace{\frac{ \ind_{ \{Y_i \le t, \delta_i=1 \}}- H^u (t|x) }{1- H(t|x)}}_{A_{1,i}(t|x)}- \underbrace{\int_0^t \frac{ \ind_{ \{Y_i \le s, \delta_i=1 \}}- H^u (s|x) }{(1- H(s|x))^2} dH(s|x)}_{A_{2,i}(t|x)} \\
         & + \underbrace{\int_0^t \frac{ \ind_{ \{Y_i \le s \}}- H (s|x) }{(1- H(s|x))^2} dH^u(s|x)}_{A_{3,i}(t|x)} + \underbrace{\frac{\left(p(Y_i,X_i)-\delta_i\right)\ind_{\{ Y_i \le t \}}}{1- H(t|x)}}_{A_{4,i}(t|x)}
    \end{split}
\end{equation*}
and direct algebra with Fubini's theorem and Cauchy-Schwarz inequality yields 
\begin{eqnarray}
\label{eq:almostsur}
    \E \left[ \left\| \widehat\Lambda_n(t|x) - \Lambda_T(t|x) \right\|_{[t_0,t_1]}^2  \right]&=& \int_a^b \E \left[ \left( \widehat\Lambda_n(t|x) - \Lambda_T(t|x) \right)^2 \right] \mathrm{d}t\\
\nonumber &\le& 5 \sum_{j=1}^4   \int_a^b    \E \left[ \left(  \sum_{i=1}^n W_h (x - X_i)A_{j,i}(t|x) \right)^2 \right] \mathrm{d}t +  5 \int_a^b \E \left[ r_n(t|x)^2 \right] \mathrm{d}t\\
\nonumber &=:& \sum_{j=1}^4 B_j+5 \int_a^b \E \left[ r_n(t|x)^2 \right] \mathrm{d}t.
\end{eqnarray}
We bound the risk of $\Lambda_n$ by controlling the $B_j$'s and the remainder term separately. Derivation of risk bounds for the $B_j$'s follow the same arguments since they are based on the bias-variance decomposition for i.i.d. random variables.\\
\bigskip

\noindent
\textbf{Part 1}: we consider the deviation results of the $R_h^x$ process defined in \eqref{eq:processRH} in order to work with weights independently distributed. Recall that $B_1$ is given by  

\begin{equation}\label{eq:A1}
    \E \left[ \left(  \sum_{i=1}^n W_h (x - X_i)A_{1,i}(t|x) \right)^2 \right]  = \frac{1}{(1- H(t|x))^2} \E \left[\left( \sum_{i=1}^n  W_h ( x -X_i)\ind_{\{Y_i \le t, \delta_i=1 \}} - H^u(t|x) \right)^2 \right].
\end{equation}
The expectation on the right hand side is in turn divided into a sum of two terms according to the value of the process $R_h^x$, the first one is 
\begin{equation*}
    \E \left[\left( \sum_{i=1}^n  W_h ( x -X_i)\ind_{\{Y_i \le t, \delta_i=1 \}} - H^u(t|x) \right)^2 \ind_{\{R_h^x \le 1/2 \}}\right] \le 4 \P\left(R^x_h \le 1/2 \right)
\end{equation*}
where we have used that $\left|\sum_{i=1}^n  W_h ( x -X_i)\ind_{\{Y_i \le t, \delta_i=1 \}}-H^u(t|x)\right| \le 2$ a.s. and by Lemma \ref{lem:deviationRhx}, we deduce that
\begin{equation}\label{eq:deviatA1}
    \begin{split}
        \E \left[\left( \sum_{i=1}^n  W_h ( x -X_i)\ind_{\{Y_i \le t, \delta_i=1 \}} - H^u(t|x) \right)^2 \ind_{\{R_h^x \le 1/2 \}}\right] & \le 8 \exp\left(-\frac{nh^p(f(x)+\psi_n^x)}{8\left(\frac{C_k^2}{c_k^2}+\frac{C_k}{2c_k}\right)}\right)
        \\ & \le \frac{1}{n h^p} \frac{64 e^{-1}\left(\frac{C_k^2}{c_k^2}+\frac{C_k}{2c_k}\right)}{c_X}
    \end{split}
\end{equation}
where $c_X>0$ is a positive constant such that $c_X<f(x)+\psi_n^x$ for $n$ large enough. For the second term, the definition of $R_h^x$ yields
\begin{equation*}
    \begin{split}
        \E  & \left[\left( \sum_{i=1}^n  W_h ( x -X_i)\ind_{\{Y_i \le t, \delta_i=1 \}} - H^u(t|x) \right)^2 \ind_{\{R_h^x \ge 1/2 \}}\right] 
        \\ &  = \E \left[\left( \frac{ \ind_{\{R_h^x \ge 1/2 \}}}{\E \left[K_h(x-X)\right] R_h^x}\left(\frac{1}{n} \sum_{i=1}^n  K_h ( x -X_i)\ind_{\{Y_i \le t, \delta_i=1 \}} -  H^u(t|x) R_h^x \E \left[K_h(x-X)\right] \right)\right)^2\right]
        \\ &  \le \E \left[\left( \frac{2}{\E \left[K_h(x-X)\right] }\left(\frac{1}{n} \sum_{i=1}^n  K_h ( x -X_i)\ind_{\{Y_i \le t, \delta_i=1 \}} -  H^u(t|x) R_h^x \E \left[K_h(x-X)\right] \right)\right)^2\right].
    \end{split}
\end{equation*}
where the bias–variance decomposition of the last term implies that we have to give upper bounds for
    \begin{eqnarray*}
        B_{t,x}^2&:=&\E \left[\frac{2}{\E \left[K_h(x-X)\right] }\left(\frac{1}{n} \sum_{i=1}^n  K_h ( x -X_i)\ind_{\{Y_i \le t, \delta_i=1 \}} -  H^u(t|x) R_h^x \E \left[K_h(x-X)\right] \right)\right]^2\\ 
    \end{eqnarray*}
and
\begin{eqnarray*}
    V_{t,x}&:=&\mathrm{Var}\left(\frac{2}{n \E \left[K_h(x-X)\right] }\sum_{i=1}^n  K_h ( x -X_i)\ind_{\{Y_i \le t, \delta_i=1 \}} \right). 
\end{eqnarray*}      
Since the random variables $\left\{(Y_i,\delta_i,X_i)\right\}_{1\le i\le n}$  are i.i.d., Lemma \ref{lem:errkernelestimH} gives a bound of the bias with
\begin{equation}\label{eq:biaisA1}
    \begin{split}
        B_{t,x} & = \frac{2}{\E \left[K_h(x-X)\right] }\E \left[\frac{1}{n} \sum_{i=1}^n  K_h ( x -X_i)\ind_{\{Y_i \le t, \delta_i=1 \}}\right] -  H^u(t|x)\E \left[K_h(x-X)\right]
        \\ &  = \frac{2}{\E \left[K_h(x-X)\right] }\left(\E \left[K_h ( x -X)\ind_{\{Y \le t, \delta=1 \}}\right] -  H^u(t|x)\E \left[K_h(x-X)\right] \right)
        \\ & \le \frac{ 2 h^{\beta} C_{K,X,\beta} }{ \E \left[K_h(x-X)\right]}
        \\ & \le h^\beta \frac{2C_{K,X,\beta}}{c_K c_X} 
    \end{split}
\end{equation}
where Lemma \ref{lem:BoundonEspK} and Lemma \ref{lem:smallballHolderversion} have been used in the last line. For the variance, we also consider the i.i.d. assumption and obtain
\begin{equation*}
        V_{t,x}  = \frac{\mathrm{Var}\left(K_h ( x -X)\ind_{\{Y \le t, \delta=1 \}} \right)}{n\E \left[K_h(x-X)\right]^2} \le \frac{ \E \left[K_h^2 ( x - X)\right]}{n\E \left[K_h(x-X)\right]^2}
\end{equation*}
where Lemma \ref{lem:BoundonEspK} and Lemma \ref{lem:smallballHolderversion} yields
    \begin{equation}\label{eq:varianceA1}
            V_{t,x}    \le \frac{C_K^2}{n c_K^2 \P(\|x-X\|\le h ) } \le \frac{C_K^2}{n h^p c_K^2 c_X  }.
    \end{equation}
Coming back to \eqref{eq:A1} and plugging \eqref{eq:deviatA1}, \eqref{eq:biaisA1} and, \eqref{eq:varianceA1} we have
\begin{equation}\label{eq:FinalBoundforA1}
    \begin{split}
    \E \left[ \left(  \sum_{i=1}^n W_h (x - X_i)A_{1,i}(t|x) \right)^2 \right] & \le  h^{2\beta} \frac{4C_{K,X,\beta}^2C_H}{c_K^2 c_X^2} +\frac{C_H}{n h^p}\left( \frac{64 e^{-1}\left(\frac{C_k^2}{c_k^2}+\frac{C_k}{2c_k}\right)}{ c_X  }+\frac{C_K^2}{c_K^2 c_X  }\right) 
    \\ & = O\left( h^{2\beta} +\frac{1}{n h^p} \right)
    \end{split}
\end{equation}
where $C_H$ is a constant such that for all $t \in [t_0,t_1]$, $\frac{1}{1-H(t|x)}\le C_H$. The proof arguments for the remaining terms are similar. By Jensen's inequality and Fubini's theorem, we have
\begin{equation*}
    \begin{split}
        \E & \left[ \left(  \sum_{i=1}^n W_h (x - X_i)A_{3,i}(t|x) \right)^2 \right] 
        \\ & = \E \left[ \left(  \sum_{i=1}^n W_h (x - X_i)\int_0^t \frac{\ind_{\{Y_i \le s  \}}-H(s|x)}{(1-H(s|x))^2}\mathrm{d}H(s|x) \right)^2 \right] 
        \\ & \le \E \left[\int_0^t  \left(  \sum_{i=1}^n W_h (x - X_i)\frac{\ind_{\{Y_i \le s  \}}-H(s|x)}{(1-H(s|x))^2}\right)^2 \mathrm{d}H(s|x)  \right] 
        \\ & \le \int_0^t \frac{1}{{(1-H(s|x))^4}} \E \left[  \left(  \sum_{i=1}^n W_h (x - X_i)\ind_{\{Y_i \le s  \}}-H(s|x)\right)^2  \right] \mathrm{d}H(s|x) .
    \end{split}
\end{equation*}
where we recognize similar terms to those in \eqref{eq:A1}. Then, as in \eqref{eq:FinalBoundforA1}, one can find a constant $C$ such that for $j=2,3,4$
\begin{equation}\label{eq:FinalBoundforA2A3}
    \E \left[ \left(  \sum_{i=1}^n W_h (x - X_i)A_{j,i}(t|x) \right)^2 \right] \le  C \left( h^{2\beta} +\frac{1}{n h^p} \right).
\end{equation}
\bigskip

\noindent
\textbf{Part 2}: recall that the reminder term is given by
\begin{equation}\label{eq:reste}
    \begin{split}
        \E \left[ r_n(t|x)^2 \right] & = \E \left[ \left( \int_0^t \frac{1}{1-H_n}-\frac{1}{1-H} \mathrm{d}\left(H_n^u-H^u \right)\right)^2\right]
        \\ & \le \E \left[ \left( \sup_{s\in [0,t]}\left|\frac{1}{1-H_n(s|x)}-\frac{1}{1-H(s|x)} \right| \int_0^t  \mathrm{d}\left( H_n^u + H^u \right)\right)^2 \right]
        \\ & \le 4 C_H^2\E \left[ \left( \sup_{s\in [0,t]}\dfrac{\left|H_n(s|x)-H(s|x) \right|}{1-H_n(s|x)} \right)^2 \right].   
        \end{split}
\end{equation}
Note here that we will not consider the discrepancies between $H_n^u$ and $H^u$ since we only need upper bounds of the same rates as the $B_j$'s and we refer to \cite{art:lo1986product,art:VanKeilegom1997Estim,Esc23-3} for more precise results about rates of convergence of $r_n$. Due to the sup-norm, we need some additional tools and consider results based on \cite{art:GineGuillou2002Rates} and firstly focus on deviations for the empirical conditional distribution function. For all $x\in\mathcal{S}_X$, we define the class of functions
\begin{equation*}
    \mathcal{G}\coloneqq \left\{ g:(u,v)\in\mathbb{R}_+\times\mathcal{S}_X \mapsto K\left(\frac{x-v}{h}\right)\ind_{\{u \le t\}};\,x\in\mathcal{S}_X,\, t>0,\,h\in(0,1)  \right\}
\end{equation*}
These functions form a uniformly bounded VC class satisfying the condition of Theorem 2.1 in \cite{art:GineGuillou2002Rates} (see the proof of Lemma 3.1 in \cite{art:EscoGoegebeurGuillouLocal} for more details) ensuring that for $\sigma^2 \ge \sup_{g\in \mathcal{F}} \mathrm{Var(g)}$, $U\geq\|g\|_\infty$ and $0<\sigma\leq U$, there exist universal constants $A$, $C$ and $L$, depending only on the VC characteristics of the class $\mathcal{G}$, such that
\begin{equation*}
    \P \left( \sup_{g\in\mathcal{G}} \left|\sum_{i=1}^n  g\left(X_i,Y_i\right) - \E \left[g\left(X,Y \right)\right]\right| > r  \right) \le L \exp\left(\frac{-r}{LU} \log \left(1 + \left(\frac{rU}{L \left(\sqrt{n}\sigma + U \sqrt{ \log\left(\frac{AU}{\sigma}\right)} \right)^2} \right)\right) \right).
\end{equation*}
whenever
\begin{eqnarray*}
    r\geq C\left[U\log\left(\dfrac{AU}{\sigma}\right)+\sqrt{n}\sigma\sqrt{\log\left(\dfrac{AU}{\sigma}\right)}\right].
\end{eqnarray*}
Here $\mathrm{Var} ( K\left(\frac{x-X}{h}\right)\ind_{\{Y\le s\}}) \leq h^p \|K\|_\infty^2$ and the choices $\sigma^2 = h^p \|K\|_{\infty}^2$ and $U= \|K\|_{\infty}$ satisfy the condition of Theorem 2.1 in \cite{art:GineGuillou2002Rates} yielding 
\begin{equation*}
    \begin{split}
        \P & \left( \sup_{g\in\mathcal{G}} \left|\sum_{i=1}^n  g\left(X_i,Y_i\right) - \E \left[g\left(X,Y \right)\right] > r  \right|\right) 
        \\ & \le L \exp\left(\frac{-r}{L \|K\|_{\infty}} \log \left(1 + \left(\frac{r \|K\|_{\infty}}{L \left(\sqrt{n h^p \|K\|_{2}^2 \|f\|_{\infty} }+ \|K\|_{\infty} \sqrt{ \log\left(\frac{A \|K\|_{\infty}}{h^p \|K\|_{2}^2 \|f\|_{\infty}}\right)} \right)^2} \right)\right) \right).
    \end{split}
\end{equation*}
One can check that for $h$ small enough, there exists a universal constant $C_{A,L}$ depending on $A$, $L$, $K$ and $f$ such that
\begin{equation*}
    \begin{split}
        \P & \left( \sup_{g\in\mathcal{G}}\left| \sum_{i=1}^n  g\left(X_i,Y_i\right) - \E \left[g\left(X,Y \right)\right]\right| > r  \right)  \le L \exp\left(-r C_{A,L}  \log \left(1 + \frac{ r C_{A,L} }{ n h^p } \right) \right).
    \end{split}
\end{equation*}
Using that for $x>0$, $\log(1+x)>x-x^2/2$, we obtain that
\begin{equation*}
    \begin{split}
        \P & \left( \sup_{g\in\mathcal{G}}\left| \sum_{i=1}^n  g\left(X_i,Y_i\right) - \E \left[g\left(X,Y \right)\right]\right| > r  \right)  \le L \exp\left(- rC_{A,L}  \left( \frac{r C_{A,L}}{n h^p }- \frac{r^2 C_{A,L}^2}{2 n^2 h^{2p} }  \right) \right),
    \end{split}
\end{equation*}
and by choosing $r=\frac{n h^p }{C_{A,L}}$, we have 
\begin{equation}
\label{eq:deviation}
    \begin{split}
        \P & \left( \sup_{g\in\mathcal{G}}\left| \sum_{i=1}^n  g\left(X_i,Y_i\right) - \E \left[g\left(X,Y \right)\right] > r  \right|\right)  \le L e^{- n h^{p}} e^{1/2} \le \frac{L e^{1/2}}{n h^p}.
    \end{split}
\end{equation}
In particular, the assumptions ($\mathcal{H}$) ensures that 
\begin{eqnarray*}
    \sup_{t\geq 0,x\in\mathcal{S}_X}\left|\mathbb{E}[K_h(x-X)\ind_{\{Y\leq t\}}]-H(t|x)f(x)\right|=\mathcal{O}(h^\beta)
\end{eqnarray*}
yielding $H_n \underset{n\to +\infty}{\longrightarrow} H$ uniformly in $t$ and $x$ a.s. This allows to reduce the inequality in (\ref{eq:reste}) to
\begin{eqnarray*}
    \E \left[ r_n(t|x)^2 \right] \le 4 C_H^4\E \left[ \left( \sup_{s\in [0,t]}\left|H_n(s|x)-H(s|x) \right| \right)^2 \right]
\end{eqnarray*}
for $n$ large enough. We thus focus on the control the right-hand term with the same idea as in the previous parts by use of the definition of the process $R_h^x$. Let us divide the expectation into two terms, the first being
\begin{equation*}
    \E \left[\left( \sup_{s\in [0,t]} \sum_{i=1}^n  W_h ( x -X_i)\ind_{\{Y_i \le s \}} - H(s|x) \right)^2 \ind_{\{R_h^x \le 1/2 \}}\right] \le 4 \P\left(R^x_h \le 1/2 \right)
\end{equation*}
where we have used that $\sup_{s\in [0,t]} \left|\sum_{i=1}^n  W_h ( x -X_i)\ind_{\{Y_i \le s \}}-H(s|x)\right| \le 2$ a.s. and by Lemma \ref{lem:deviationRhx},
\begin{equation}\label{eq:deviatA1part5}
        \E \left[\left( \sup_{s\in [0,t]}\left|\sum_{i=1}^n  W_h ( x -X_i)\ind_{\{Y_i \le s \}} - H(s|x) \right| \right)^2 \ind_{\{R_h^x \le 1/2 \}}\right]  \le \frac{64 e^{-1}\left(\frac{C_k^2}{c_k^2}+\frac{C_k}{2c_k}\right)}{nh^p c_X}.
\end{equation}
For the second term, we re-use the definition of $R_h^x$ by adding and subtracting $\frac{\E \left[ K_h(x-X) \ind_{\{Y\le s \}}\right]}{\E \left[ K_h(x-X)\right]R_h^x} $ in the absolute value, which by the triangular inequality gives 
\begin{equation*}
    \begin{split}
          \E & \left[ \left( \sup_{s\in [0,t]} \left| \sum_{i=1}^n  W_h ( x -X_i)\ind_{\{Y_i \le s \}} - H(s|x) \right| \right)^2 \ind_{\{R_h^x \ge 1/2 \}}\right]
        \\ & = \E \left[ \left( \sup_{s\in [0,t]} \left| \frac{ 1 }{n \E \left[K_h(x-X)\right] R_h^x} \sum_{i=1}^n  K_h ( x -X_i)\ind_{\{Y_i \le s\}} -  H(s|x) \right|\right)^2 \ind_{\{R_h^x \ge 1/2 \}}\right]
        \\ &  \le  2\E \left[ \left( \sup_{s\in [0,t]} \left| \frac{ 1 }{n \E \left[K_h(x-X)\right] R_h^x} \sum_{i=1}^n  \left( K_h ( x -X_i)\ind_{\{Y_i \le s\}} - \E \left[ K_h(x-X) \ind_{\{Y\le s \}}\right]\right) \right| \right)^2 \ind_{\{R_h^x \ge 1/2 \}} \right]
        \\ & \qquad \qquad +  2\E \left[ \left( \sup_{s\in [0,t]}  \left|  \frac{ 1 }{n \E \left[K_h(x-X)\right] R_h^x} \E \left[ K_h(x-X) \ind_{\{Y\le s \}}\right]  - H(s|x) \right| \right)^2 \ind_{\{R_h^x \ge 1/2 \}}\right]
        \\ &  \le \underbrace{\frac{ 8 }{ \E \left[K_h(x-X)\right]^2}  \E \left[ \left( \sup_{s\in [0,t]} \left| \frac{ 1 }{n} \sum_{i=1}^n  \Big( K_h ( x -X_i)\ind_{\{Y_i \le s\}} - \E \left[ K_h(x-X) \ind_{\{Y\le s \}}\right]\Big) \right| \right)^2 \right]}_{\eqqcolon E_1}
        \\ & \qquad \qquad +  \underbrace{2\E \left[ \left( \sup_{s\in [0,t]}  \left|  \frac{ 1 }{ \E \left[K_h(x-X)\right] R_h^x} \E \left[ K_h(x-X) \ind_{\{Y\le s \}}\right]  - H(s|x) \right| \right)^2 \ind_{\{R_h^x \ge 1/2 \}}\right]}_{\eqqcolon E_2},
    \end{split}
\end{equation*}
By lemmas \ref{lem:BoundonEspK} and \ref{lem:smallballHolderversion}, we first have
\begin{equation*}
        E_1 \le \frac{8}{c_K^2 c_X^2}  \E \left[ \left( \sup_{s\in [0,t]} \left| \frac{ 1 }{n} \sum_{i=1}^n  \Big( K_h ( x -X_i)\ind_{\{Y_i \le s\}} - \E \left[ K_h(x-X) \ind_{\{Y\le s \}}\right]\Big) \right| \right)^2 \right].
\end{equation*}
The upper bound value of $E_1$ is in turn divided into two terms. For all $r>0$, denote the random event 
\begin{eqnarray*}
    \mathcal{E}_r:=\left\{\sup_{s\in [0,t]} \left|\sum_{i=1}^n  g\left(X_i,Y_i\right) - \E \left[g\left(X,Y \right)\right]\right| \leq r \right\}
\end{eqnarray*}
then we have that $E_1$ is bounded above by
\begin{eqnarray*}\label{eq:E1origine}
            &&\frac{8}{n^2 c_K^2 c_X^2} \left(\E \left[ \left( \sup_{s\in [0,t]} \left| \sum_{i=1}^n  g\left(X_i,Y_i\right) - \E \left[g\left(X,Y \right)\right]\Big) \right| \right)^2 \ind_{\mathcal{E}_r} \right]\right.\\
            &&+ \left.\E \left[ \left( \sup_{s\in [0,t]} \left| \sum_{i=1}^n  g\left(X_i,Y_i\right) - \E \left[g\left(X,Y \right)\right]\Big) \right| \right)^2 \ind_{(\mathcal{E}_r)^c} \right]\right)\\
            & \le &\frac{8 }{ c_K^2 c_X^2}\left(  \frac{r^2}{n^2} + 4 \P \left( \sup_{s\in [0,t]} \sum_{i=1}^n  g\left(X_i,Y_i\right) - \E \left[g\left(X,Y \right)\right] > r  \right)\right)
\end{eqnarray*}
and according to (\ref{eq:deviation}), we obtain
\begin{equation}\label{eq:E1final}
    E_1 \le \frac{8 }{ c_K^2 c_X^2}\left(  \frac{h^{2p}}{ C_{A,L}} + 4 \frac{L e^{1/2}}{n h^p}\right) .
\end{equation}
Adding and subtracting $ \frac{H(s|x)}{R_h^x} $, the term $E_2$ is in turn divided into two terms
\begin{equation*}
    \begin{split}
        E_2 & \le  2\E \left[ \left( \sup_{s\in [0,t]}  \left|  \frac{ 1 }{ \E \left[K_h(x-X)\right] R_h^x} \left(\E \left[ K_h(x-X) \ind_{\{Y\le s \}}\right]  - \E \left[K_h(x-X)\right]H(s|x) \right) \right| \right)^2 \ind_{\{R_h^x \ge 1/2 \}}\right] 
        \\ & \qquad \qquad + 2\E \left[ \left( \sup_{s\in [0,t]}  \left|  \frac{ 1 }{R_h^x} \left(H(s|x) -  H(s|x)R_h^x  \right) \right| \right)^2 \ind_{\{R_h^x \ge 1/2 \}}\right] 
        \\ & \le  \underbrace{2\E \left[ \left( \sup_{s\in [0,t]}  \left|  \frac{ 2 }{ \E \left[K_h(x-X)\right]} \left( \E \left[ K_h(x-X) \ind_{\{Y\le s \}}\right]  - \E\left[ K_h(x-X)\right]H(s|x)  \right)\right| \right)^2 \right] }_{\eqqcolon E_2^{(1)}}
        \\ & \qquad \qquad + \underbrace{2\E \left[ \left( \sup_{s\in [0,t]}  \left|  2 H(s|x) (1-R^h_x) \right| \right)^2 \right] }_{\eqqcolon E_2^{(2)}}.
    \end{split}
\end{equation*}
We remark that $E_2^{(1)}$ is no longer random, moreover by Lemma \ref{lem:errkernelestimH} we get 
\begin{equation*}
    E_2^{(1)} \le \frac{8h^{2\beta}C_{K,X,\beta}^2}{ \E \left[K_h(x-X)\right]^2}, 
\end{equation*}
which by use of Lemma \ref{lem:BoundonEspK} and Lemma \ref{lem:smallballHolderversion} gives
\begin{equation}\label{eq:E21}
    E_2^{(1)} \le \frac{8 h^{2\beta}C_{K,X,\beta}^2}{c_K^2 c_X^2}.
\end{equation}
For the term $E_2^{(2)}$, we bound $H(s|x)$ by one and see that the remaining does not depend on $s$ anymore. Moreover, since $\E \left[R_h^x\right]=1$, we have
\begin{equation*}
     E_2^{(2)} \le 8 \mathrm{Var} \left(R_h^x \right) \le  \frac{8 \E \left[K_h(x-X)^2\right]}{n \E \left[K_h(x-X)\right]^2},
\end{equation*}
Again thanks to lemmas \ref{lem:BoundonEspK} and \ref{lem:smallballHolderversion}, we get
\begin{equation}\label{eq:E22}
     E_2^{(2)} \le \frac{8C_K^2}{n h^p c_K^2 c_X} .
\end{equation}
Plugging \eqref{eq:deviatA1part5}, \eqref{eq:E1final}, \eqref{eq:E21} and \eqref{eq:E22} into \eqref{eq:reste} we have 
\begin{equation}\label{eq:FinalBoundforA5}
    \begin{split}
        \E & \left[ r_n(t|x)^2 \right] 
        \\ & \le 8C_H \left(\frac{1}{n h^p} \frac{64 e^{-1}\left(\frac{C_k^2}{c_k^2}+\frac{C_k}{2c_k}\right)}{ c_X  } + \frac{4 }{ c_K^2 c_X^2}\left(  \frac{h^{2p}}{ C_{A,L}} + 4 \frac{L e^{1/2}}{n h^p}\right) + \frac{4 h^{2\beta}C_{K,X,\beta}^2}{c_K^2 c_X^2} + \frac{4C_K^2}{n h^p c_K^2 c_X}  \right)
        \\ & \le C \left( h^{2 \beta} + \frac{1}{n h^p} \right),
    \end{split}
\end{equation}
where in the last line, we have used the fact that $h^p\le h^\beta$.\\ 

\subsubsection{Risk analysis for $\widehat\Lambda_n-\Lambda_n$}

By definition, the statistic $\widehat\Lambda_n-\Lambda_n$ is given by
\begin{eqnarray*}
    \int_0^t\dfrac{d[\widehat H^u_n-H^u_n](s|x)}{1-H_n(s|x)}&=&\sum_{i=1}^nW_h(x-X_i)\dfrac{[\widehat p_n-p](X_i,Y_i))\ind_{\{Y_i\leq t\}}}{1-H_n(Y_i|x)}\\
    &=&\sum_{i=1}^n\dfrac{W_h(x-X_i)\ind_{\{Y_i\leq t\}}}{(1-H_n(Y_i|x))}\sum_{j=1}^n\widetilde W_{b}((Y_i,X_i)-(Y_j,X_j))(\delta_j-p(X_i,Y_i)+\mu_j).
\end{eqnarray*}
For the ease of reading, we denote for any $i,j=1\ldots,n$
\begin{eqnarray*}
    \kappa_i:=K_h(x-X_i)\ind_{\{Y_i\leq t\}},\quad\quad\kappa_{i,j}:=K_{b}(X_i-Y_j)\widetilde K_b(Y_i-Y_j)
\end{eqnarray*}
and
\begin{eqnarray*}
    \omega_i:=W_h(x-X_i)\ind_{\{Y_i\leq t\}},\quad\quad\omega_{i,j}:=\widetilde W_{b}((Y_i,X_i)-(Y_j,X_j))
\end{eqnarray*}
so that
\begin{eqnarray*}
    \int_0^t\dfrac{d[\widehat H^u_n-H^u_n](s|x)}{1-H_n(s|x)}=\sum_{i=1}^n\dfrac{\omega_i}{(1-H_n(Y_i|x))}\sum_{j=1}^n\omega_{i,j}(\delta_j-p(X_i,Y_i)+\mu_j).
\end{eqnarray*}
Furthermore, we define
\begin{eqnarray*}
    R_{b,-i}^{y,x}:=\dfrac{1}{n-1}\sum_{j=1,j\neq i}\dfrac{K_b(x-X_j)\widetilde K_b(y-Y_j)}{\E[K_b(x-X)\widetilde K_b(y-Y)]},\quad i=1,\ldots,n.
\end{eqnarray*}
which allows to link the random variables introduced above with
\begin{eqnarray*}
    \omega_i=\dfrac{\kappa_i}{n\E[K_h(x-X)]R_h^x}
\end{eqnarray*}
and
\begin{eqnarray*}
    \omega_{i,j}=\dfrac{\kappa_{i,j}}{(n-1)\E_i[K_b(X_i-X)\widetilde K_b(Y_i-Y)]R_{b,-i}^{Y_i,X_i}+K_b(0)\widetilde K_b(0)}
\end{eqnarray*}
where $\E_i$ denotes the expectation conditionally on $(X_i,Y_i)$. The computation of the risk upper bound in this section is twofold. In the first part, we study the consistency of the random weights and in the second part, prove that the risk admits a convergence rate with order $1/\alpha^2nb^{1+p}$.\\  

\noindent
\textbf{Part 1} : to ease the forthcoming analysis, we consider a partition of the probability space based on the events $A_i:=\{R_h^x\geq 1/2,\, R_{b,-i}^{Y_i,X_i}\geq 1/2\}$ and show that statistics over the sets $A_i^c$ are negligible. By inclusion of the random events $A_i^c$ in $\{R_h^x\geq 1/2\}\cup\{R_{b,-i}^{Y_i,X_i}\geq 1/2\}$, we have
\begin{eqnarray*}
    \P(A_i^c)\leq \P(R_h^x\leq 1/2)+\P(R_{b,-i}^{Y_i,X_i}\leq 1/2)
\end{eqnarray*}
which by Jensen's inequality gives us
\begin{eqnarray*}
   \E\left[\left(\sum_{i=1}^n\dfrac{\omega_i\ind_{A_i^c}}{(1-H_n(Y_i|x))}\sum_{j=1}^n\omega_{i,j}(\delta_j-p(X_i,Y_i)+\mu_j)\right)^2\right]&\leq& C_H^2\E\left[\sum_{i=1}^n\omega_i\ind_{A_i^c}\sum_{j=1}^n\omega_{i,j}\left(\delta_j-p(X_i,Y_i)+\mu_j\right)^2\right]\\
   &=&C_H^2\E\left[\sum_{i=1}^n\omega_i\ind_{A_i^c}\sum_{j=1}^n\E_i\left[\omega_{i,j}(\delta_j-p(X_i,Y_i)+\mu_j)^2\right]\right].
\end{eqnarray*}
Note that by independence between $\{\mu_i\}_{1\leq i\leq n}$ and the survival data, we have for any $i=1,\ldots,n$
\begin{eqnarray*}
    \E_i\left[\omega_{i,j}(\delta_j-p(X_i,Y_i)+\mu_j)^2\right]=\E_i\left[\omega_{i,j}(\delta_j-p(X_i,Y_i))^2\right]+\dfrac{1}{\alpha^2}\E_i[\omega_{i,j}]
\end{eqnarray*}
which implies that
\begin{eqnarray*}
    \E\left[\sum_{i=1}^n\omega_i\ind_{A_i^c}\sum_{j=1}^n\E_i\left[\omega_{i,j}(\delta_j-p(X_i,Y_i)+\mu_j)^2\right]\right]\leq C_H^2(1+\alpha^{-2})\E\left[\sum_{i=1}^n\omega_i\ind_{A_i^c}\right].
\end{eqnarray*}
In the particular, the right-hand expectation is bounded up to
\begin{eqnarray}
\label{eq::boundriskprivate1}
    \nonumber\P(R_h^x\leq 1/2)+\sum_{i=1}^n\mathbb{E}\left[\omega_i\ind_{\{R_{b,-i}^{Y_i,X_i}\leq 1/2\}}\right]&\leq&e^{-c_1nh^p}+\sum_{i=1}^n\mathbb{E}\left[\omega_i\P_i\left(R_{b,-i}^{Y_i,X_i}\leq 1/2\right)\right]\\
    &\leq& e^{-c_1nh^p}+\sum_{i=1}^n\mathbb{E}\left[\omega_ie^{-c_2nb^{1+p}\left(f(X_i)g(Y_i|X_i)+\phi^{X_i,Y_i}_n\right)}\right]
\end{eqnarray}
where the $c_i$'s are positive constants derived from Lemma \ref{lem:deviationRhx}. Furthermore, let $c_3$ be a positive constant such that $\sup_{x,y}|\phi^{x,y}_n|\leq c_3<\inf_{y\leq b}g(y|x)f(x)$, then
\begin{eqnarray*}
    &&\sum_{i=1}^n\mathbb{E}\left[\omega_i\exp\left(-c_2nb^{1+p}(f(X_i)g(Y_i|X_i)+\phi^{X_i,Y_i}_n)\right)\right]\\
    &\leq&e^{-c_1nh^p}+\sum_{i=1}^n\mathbb{E}\left[\omega_i\exp\left(-c_2nb^{1+p}(f(X_i)g(Y_i|X_i)-c_3)\right)\ind_{\{R_h^x\geq 1/2\}}\right]\\
    &\leq&e^{-c_1nh^p}+\dfrac{2}{n(f(x)+\psi^x_n)c_K}\sum_{i=1}^n\mathbb{E}\left[\kappa_i\exp\left(-c_2nb^{1+p}(f(X_i)g(Y_i|X_i)-c_3)\right)\ind_{\{R_h^x\geq 1/2\}}\right]
\end{eqnarray*}
where
\begin{eqnarray*}
    &&\mathbb{E}\left[\kappa_i\exp\left(-c_2nb^{1+p}(f(X_i)g(Y_i|X_i)-c_3)\right)\ind_{\{R_h^x\geq 1/2\}}\right]\\
    &\leq&\int_{B(0,1)}K(u)f(x-uh)\int_0^tg(v|x-uh)\exp\left(-c_2nb^{1+p}(f(x-uh)g(v|x-uh)-c_3)\right)dvdu.
\end{eqnarray*}
Due to the model assumptions, we can also find a positive constant $c_4>0$ such that
\begin{eqnarray}
\label{eq::fgbound}
    f(x-uh)g(v|x-uh)\geq f(x)g(v|x)-c_4h^\beta\geq f(x)\inf_{y\leq b}g(y|x)-c_4h^\beta
\end{eqnarray}
implying
\begin{eqnarray*}
    &&\int_{B(0,1)}K(u)f(x-uh)\int_0^tg(v|x-uh)\exp\left(-c_2nb^{1+p}(f(x-uh)g(v|x-uh)-c_3)\right)dvdu\\
    &\leq&\exp\left(-c_2nb^{1+p}(f(x)\inf_{y\leq b}g(y|x)-c_3-c_4h^\beta)\right)\|f\|_\infty.
\end{eqnarray*}
Back to (\ref{eq::boundriskprivate1}), this ensures that
\begin{eqnarray}
\label{eq::boundriskprivate2}
    \nonumber&&\E\left[\left(\sum_{i=1}^n\dfrac{\omega_i\ind_{A_i^c}}{(1-H_n(Y_i|x))}\sum_{j=1}^n\omega_{i,j}(\delta_j-p(X_i,Y_i)+\mu_j)\right)^2\right]\\
    &\lesssim&(1+\alpha^{-2})\left(e^{-c_1nh^p}+e^{-c_2nb^{1+p}(f(x)\inf_{y\leq b}g(y|x)-c_3-c_4h^\beta)}\right)\lesssim\dfrac{1+\alpha^{-2}}{nh^p}.
\end{eqnarray}

\noindent
\textbf{Part 2} : according to part 1, we can assume $R_h^x\geq 1/2$ and $R_{b,-i}^{Y_i,X_i}\geq 1/2$ for any $i=1,\ldots,n$ without lost of generality. By Jensen's inequality and Lemma \ref{lem:smallballHolderversion} and \ref{lem:BoundonEspK}, we have
\begin{eqnarray*}
     \mathbb{E}\left[\left(\sum_{i=1}^n\omega_i\dfrac{[\widehat p_n-p](X_i,Y_i)}{1-H_n(Y_i|x)}\right)^2\right]\leq\dfrac{8C_H^2}{n(n-1)^2c_K^3c_{\widetilde K}^2(f(x)+\phi_n^x)} \sum_{i=1}^n\mathbb{E}\left[\kappa_i\dfrac{\E_i\left[\left(\sum_{j=1}^n\kappa_{i,j}[\delta_j-p(X_i,Y_i)+\mu_j]\right)^2\right]}{\left(f(X_i)g(Y_i|X_i)+\phi_n^{X_i,Y_i}+\varepsilon_n\right)^2}\right].
\end{eqnarray*}
Here, we compute the conditional expectation through a sum of i.i.d. centred random variables, in the sense that for any $i=1,\ldots,n$
\begin{eqnarray*}
    \E_i\left[\left(\sum_{j=1}^n \kappa_{i,j}[\delta_j-p(X_i,Y_i)+\mu_j]\right)^2\right]&=&\E_i\left[\left(\sum_{j=1,j\neq i}^n\kappa_{i,j}[\delta_j-p(X_i,Y_i)+\mu_j]\right)^2\right]\\
    &&+\left(K_b(0)\widetilde K_b(0)[\delta_i-p(X_i,Y_i)+\mu_i]\right)^2\\
    &=&(n-1)\E_i\left[\left(K_b(X_i-X)\widetilde K_b(Y_i-Y)[\delta-p(X_i,Y_i)+\mu]\right)^2\right]\\
    &&+\left(K_b(0)\widetilde K_b(0)[\delta_i-p(X_i,Y_i)+\mu_i]\right)^2
\end{eqnarray*}
where $(X,Y,\delta,\mu)$ is an independent random vector drawn from the survival model. Note here that Lemma \ref{lem:smallballHolderversion} and \ref{lem:BoundonEspK} also show that
\begin{eqnarray*}
   \E_i\left[\left(K_b(X_i-X)\widetilde K_b(Y_i-Y)[\delta-p(X_i,Y_i)+\mu]\right)^2\right]&\leq&(1+\alpha^{-2})\E_i\left[\left(K_b(X_i-X)\widetilde K_b(Y_i-Y)\right)^2\right]\\
   &\lesssim&(1+\alpha^{-2})b^{-(1+p)}\left(f(X_i)g(Y_i|X_i)+\phi_n^{X_i,Y_i}\right)
\end{eqnarray*}
implying
\begin{eqnarray*}
    &&\E_i\left[\left(\sum_{j=1}^n \kappa_{i,j}[\delta_j-p(X_i,Y_i)+\mu_j]\right)^2\right]\lesssim (1+\alpha^{-2})(n-1)b^{-(1+p)}\left(f(X_i)g(Y_i|X_i)+\phi_n^{X_i,Y_i}\right)+b^{-2(1+p)}.
\end{eqnarray*}
Hence
\begin{eqnarray*}
    &&\sum_{i=1}^n\mathbb{E}\left[\dfrac{\kappa_i\E_i\left[\left(\sum_{j=1}^n \kappa_{i,j}[\delta_j-p(X_i,Y_i)+\mu_j]\right)^2\right]}{\left(f(X_i)g(Y_i|X_i)+\phi_n^{X_i,Y_i}+\varepsilon_n\right)^2}\right]\\
    &\lesssim&(1+\alpha^{-2})n\E\left[K_h(x-X)\dfrac{(n-1)b^{-(1+p)}\left(f(X_i)g(Y_i|X_i)+\phi_n^{X_i,Y_i}\right)+b^{-2(1+p)}}{\left(f(X_i)g(Y_i|X_i)+\phi_n^{X_i,Y_i}+\varepsilon_n\right)^2}\ind_{\{Y\leq t\}}\right]\\
    &=&(1+\alpha^{-2})\dfrac{n(n-1)}{b^{1+p}}\E\left[K_h(x-X)\dfrac{f(X)g(Y|X)+\phi_n^{X,Y}+(nb^{1+p})^{-1})}{\left(f(X)g(Y|X)+\phi_n^{X,Y}+\varepsilon_n\right)^2}\ind_{\{Y\leq t\}}\right].
\end{eqnarray*}
Since the kernel functions have support in the unit ball, the above expectation is strictly positive whenever $X\in B(x,h)$. By (\ref{eq::fgbound}) this implies that for $n$ large enough, we have
\begin{eqnarray*}
    \dfrac{f(X)g(Y|X)+\phi_n^{X,Y}+(nb^{1+p})^{-1})}{\left(f(X)g(Y|X)+\phi_n^{X,Y}+\varepsilon_n\right)^2}\leq 4\dfrac{\| f\|_\infty \|g\|_\infty}{\left(f(x)\inf_{y\leq t_1}g(y|x)\right)^2},\quad a.s.
\end{eqnarray*}
which ensures with (\ref{eq::boundriskprivate2}) that
\begin{eqnarray}
\label{eq::riskpn2}
    &&\mathbb{E}\left[\left(\sum_{i=1}^nW_h(x-X_i)\dfrac{[p_n-p](X_i,Y_i)\ind_{\{Y_i\leq t\}}}{1-H_n(Y_i|x)}\right)^2\right]\lesssim \dfrac{1+\alpha^{-2}}{nb^{1+p}}.
\end{eqnarray}
Finally, the combined results in \eqref{eq:FinalBoundforA1}, \eqref{eq:FinalBoundforA2A3}, \eqref{eq:FinalBoundforA5} and (\ref{eq::riskpn2}) in \eqref{eq:almostsur} conclude the proof with
\begin{equation*}
    \E \left[ \left\| \widehat\Lambda_n(t|x) - \Lambda_T(t|x) \right\|_{[t_0,t_1]}^2  \right] \le  C (t_1-t_0) \left(h^{2\beta}+ \frac{1+\alpha^{-2}}{n  h^p}  \right).
\end{equation*}
\hfill\qed

\subsection{Proof of Theorem \ref{theorem:minimax}}
In this section, we denote by $F_i(\cdot |x)$ the conditional distribution function of a random variable $T_i$ given $x$ and $\Lambda_i(\cdot | x)$ its conditional cumulative hazard function.\\

\noindent
First, for a fixed $\kappa>0$, thanks to the Markov inequality, the problem remains to lower bound 
\begin{equation*}
    \inf_{\tilde Q \in \mathcal{\tilde Q_\alpha}}\inf_{ \widetilde \Lambda} \sup_{P \in \mathcal{P}_{\eta,\beta}} \E_{P,\tilde Q} \left[  \left\| \widetilde{\Lambda}_n(.|x)- \Lambda_{T(P)}(.|x)\right\|^2_{[t_0,t_1]} \right] \ge \kappa^2\underbrace{\inf_{\tilde Q \in \mathcal{\tilde Q_\alpha}}\inf_{ \widetilde \Lambda} \sup_{P \in \mathcal{P}_{\eta,\beta}}P \left( \left\| \widetilde{\Lambda}_n(.|x)- \Lambda_{T(P)}(.|x)\right\|_{[t_0,t_1]} \ge \kappa \right)}_{\eqqcolon \mathfrak{M}  }.
\end{equation*}
We also define this quantity for a fixed $\tilde Q$
\begin{equation*}
    \mathfrak{M}_{\tilde Q} \coloneqq \inf_{ \widetilde \Lambda} \sup_{P \in \mathcal{P}_{\eta,\beta}}P \left( \left\| \widetilde{\Lambda}_n(.|x)- \Lambda_{T(P)}(.|x)\right\|_{[t_0,t_1]} \ge \kappa \right).
\end{equation*}
The standard idea is to reduce the problem to a testing problem (see \cite{book:Tsybakov_nonparam} and \cite{art:Duchi_Jordan_Wainwright_MinimaxLocalPrivat} for $\alpha$-locally private setting). In particular, the Le Cam's bounds consider two distributions $P_0, P_1\in\mathcal{P}_{\eta,\beta}$ satisfying a $2\kappa$-separated assumption:
\begin{equation}\label{hyp:2kappa}
    \left\| \Lambda_0(\cdot|x_0) - \Lambda_1(\cdot|x_0) \right\|_{[t_0,t_1]}\ge 2 \kappa .
\end{equation}
To satisfy this inequality, we have to choose two random variables $P_0$ and $P_1$ in $\mathcal{P}_\beta$ satisfying the $2\kappa$-separated assumption. So that, we define  two triples $(T_0,C_0,X_0)$ and $(T_1,C_1,X_1)$ where $C_0$, $C_1$, $X_0$, $X_1$ will be chosen later. For $T_0$ and $T_1$, as in \cite{art:ChagnyRoche_AdaptiveMinimax} we define $T_0$ as an uniform law on $[t_0,t_1]$ (note that $P_0$ does not depend on $x$). We define $T_1$ by its conditional distribution function
\begin{equation*}
    F_1(t|x)= F_0(t) + \eta_n^\beta I\left(\frac{\|x-x_0\|}{\eta_n}\right)\int_{-\infty}^t \psi(s) \mathrm{d}s,
\end{equation*}
where $\eta_n$ is a non-negative number chosen later, $\psi$ is a compactly supported function on $[t_0,t_1]$ such that $\int_\R \psi(y) \mathrm{d}y = 0$ and $I : \R_+ \to \R $ is a compactly supported function on $[0,1]$ such that there exists $c\in \R_+$, $|I(x)-I(y)| \le c |x-y|^\beta$ for all $(x,y) \in \R_+^2$. For convenient reasons, and to explicit the constants, we define for now the function $\psi$ by the following formula: for all $t \in \R$
\begin{equation*}
    \psi (t) = (t - t_0/2 - t_1/2)\ind_{[t_0,t_1]}(t).
\end{equation*}
Thanks to the Hölder-type regularity of $I$, one can check that the map $x \mapsto F_1(t|\cdot)$ also has this property. To check the $2\kappa$-separated assumption \eqref{hyp:2kappa},  observe that the mean value theorem implies for all $t \in [t_0,t_1]$ and $x\in \R^d$:
\begin{equation*}
    \begin{split}
     | F_0(t)-F_1(t|x) | = | e^{-\Lambda_0(t|x)}- e^{-\Lambda_1(t|x)} |&  \le \sup_{s \in \Lambda_0([t_0,t_1])\cup \Lambda_1([t_0,t_1])}e^{-s} | \Lambda_0(t|x)-\Lambda_1(t|x) |
     \\ & \le   | \Lambda_0(t|x)-\Lambda_1(t|x) |.
    \end{split}
\end{equation*}
This implies 
\begin{equation*}
    \| \Lambda_0(\cdot|x)-\Lambda_1(\cdot|x) \|_{[t_0,t_1]} \ge \| F_0(\cdot|x)-F_1(\cdot|x) \|_{[t_0,t_1]}.
\end{equation*}
As a result, we check the $2\kappa$-separated assumption \eqref{hyp:2kappa} on the distribution functions for a fixed $x_0$. We have 
\begin{equation*}
    \begin{split}
    \| F_0(\cdot|x_0)-F_1(\cdot|x_0) \|_{[t_0,t_1]} & = \eta_n^\beta I(0) \left(\int_{t_0}^{t_1} \left(\int_{-\infty}^t \psi(s) \mathrm{d}s \right)^2 \mathrm{d}t \right)^{1/2}
    \\ & = \eta_n^\beta I(0) \left(\int_{t_0}^{t_1} \left(\frac{(t-t_1)(t-t_0)}{2}\right)^2 \mathrm{d}t \right)^{1/2}
    \\ & = \frac{\eta_n^\beta I(0)  (t_1-t_0)^{5/2}}{2\sqrt{30}}.
    \end{split}
\end{equation*}
The choice $\eta_n^\beta=\frac{4\kappa \sqrt{30} }{I(0)  (t_1-t_0)^{5/2}}$ guarantees the $2\kappa$-separated inequality \eqref{hyp:2kappa}. Now, under this assumption \cite[sec. 2.2]{book:Tsybakov_nonparam} ensure that, 
\begin{equation*}
    \mathfrak{M}_{\tilde Q} \ge \frac{1}{2}\left(1-\| M_0^n - M_1^n \|_{TV} \right)
\end{equation*}
where $M_i^n$ is the joint law of $\alpha$-privatized mechanism $\widetilde Q$ of $n$ observations $\left(P_{i,1}^{(obs)}, \dots, P_{i,n}^{(obs)} \right)$ where each $P_{i,j}^{(obs)}$ is distributed according to $\mathcal{P}^{(obs)}_\beta$. By definition, $M_0$ and $M_1$ are absolutely continuous with respect to the same law with densities denoted $m_0$ and $m_1$ respectively. Then, one can use the Pinsker inequality 
\begin{equation}\label{eq:privateDKL}
    \begin{split}
        \mathfrak{M}_Q & \ge \frac{1}{2}\left(1-\frac{1}{2} \sqrt{ D_{\mathrm{kl}}(M_0^n \| M_1^n)+D_{\mathrm{kl}}(M_1^n \| M_0^n)}\right)
        \\ & \ge \frac{1}{2}\left(1-\frac{\sqrt{n}}{2}\sqrt{ D_{\mathrm{kl}}(M_0 \| M_1)+D_{\mathrm{kl}}(M_1 \| M_0)}\right).
    \end{split}
\end{equation}
It remains to upper bound the sum of the two divergences, we proceed as \cite[App A.1]{art:Duchi_Jordan_Wainwright_MinimaxLocalPrivat}: by definition of the Kullback-Leibler divergence,
\begin{equation} \label{eq:DKL}
    \begin{split}
     D_{\mathrm{kl}}(M_0 \| M_1)& +D_{\mathrm{kl}}(M_1 \| M_0)
     \\ & = \int m_0(t,z,x)\log \left(\frac{m_0(t,z,x)}{m_1(t,z,x)}\right) \mathrm{d}z \mathrm{d}t \mathrm{d}\P_X(x)+ \int m_1(t,z,x)\log \left(\frac{m_1(t,z,x)}{m_0(t,z,x)}\right) \mathrm{d}z \mathrm{d}t \mathrm{d}\P_X(x)
    \\ & = \int \left(m_0(t,z,x)-m_1(t,z,x)\right)\log \left(\frac{m_0(t,z,x)}{m_1(t,z,x)}\right) \mathrm{d}z \mathrm{d}t \mathrm{d}\P_X(x)
    \\ & \le \int \frac{\left(m_0(t,z,x)-m_1(t,z,x)\right)^2}{\min(m_0(t,z,x),m_1(t,z,x))}\mathrm{d}z \mathrm{d}t \mathrm{d}\P_X(x), 
    \end{split}
\end{equation}
where in the last line we used \cite[App A, Lemma 4]{art:Duchi_Jordan_Wainwright_MinimaxLocalPrivat}.
For now, we define the censor law of our two models by $C_1=T_0$ and $C_0=T_1$. Thanks to this choice, the law of the observation time (the minimum between the real survival time $T$ and the censor $C$) is the same. This makes the computation of the difference between the densities of the observations much smaller. Indeed, by definition, we have
\begin{equation*}
    \begin{split}
        & m_0(t,z,x)- m_1(t,z,x)  = \sum_{i=0}^1  q(z|i) \left(  f_0^{(o b s)}(t,i,x) -  f_1^{(obs)}(t,i,x) \right)  
        \\ & = \sum_{i=0}^1  q(z|i) \left(\left( \bar F_{T_1}(t|x) f_{T_0}(t|x)\right)^{i} \left( \bar F_{T_0}(t|x) f_{T_1}(t|x) \right)^{1-i} - \left(\bar F_{T_0}(t|x) f_{T_1}(t|x)\right)^{i} \left( \bar F_{T_1}(t|x) f_{T_0}(t|x) \right)^{1-i} \right)  
        \\ & = \sum_{i=0}^1   q(z|i) \underbrace{\left(\left( \bar F_{T_1}(t|x) f_{T_0}(t|x) - \bar F_{T_0}(t|x) f_{T_1}(t|x) \right)^{i}\left( \bar F_{T_0}(t|x) f_{T_1}(t|x) - \bar F_{T_1}(t|x) f_{T_0}(t|x) \right)^{1-i} \right)}_{\coloneqq h(t,i,x)}  .
    \end{split}
\end{equation*}
Following the idea of Lemma $3$ of \cite{art:Duchi_Jordan_Wainwright_MinimaxLocalPrivat}, let us decompose the function $h$ into its negative and positive parts,
\begin{equation*}
    \begin{split}
        & m_0(t,z,x)- m_1(t,z,x)  = \sum_{i=0}^1  q(z| i)h^+(t,i,x)   +    \sum_{i=0}^1  q(z|i)  h^-(t,i,x) 
        \\ & \le \sup_{\delta \in \{ 0,1\}} q(z|\delta)   \sum_{i=0}^1 h^+(t,i,x)  + \inf_{\delta \in \{ 0,1\}} q(z|\delta)  \sum_{i=0}^1h^-(t,i,x)
        \\ & \le \left(\sup_{\delta \in \{ 0,1\}} q(z|\delta) - \inf_{\delta\in \{ 0,1\}} q(z|\delta)   \right)  \sum_{i=0}^1 h^+(t,i,x),
    \end{split}
\end{equation*}
where in the last line, we observe that $\sum_{i=0}^1 h^+ (t,i,x)  = - \sum_{i=0}^1 h^-(t,i,x) $ since $\sum_{i=0}^1 h =0$. We also remark that $\sum_{i=0}^1 h^+(t,\delta,x) = \left|\bar F_{T_1}(t|x) f_{T_0}(t|x) - \bar F_{T_0}(t|x) f_{T_1}(t|x) \right|$. Moreover, with the help of the definition of $\alpha$-local differential privacy for the indicator censor, we have:
\begin{equation*}
    \begin{split}
        & |m_0(t,z,x)- m_1(t,z,x) | \le \left| \inf_{\delta \in \{ 0,1\}} q(z|\delta) \left(\frac{\sup_{\delta\in \{ 0,1\}} q(z|\delta)}{ \inf_{\delta' \in \{ 0,1\}} q(z|\delta') } - 1 \right) \right|\left|\bar F_{T_1}(t|x) f_{T_0}(t|x) - \bar F_{T_0}(t|x) f_{T_1}(t|x) \right|
        \\ & \le   \inf_{\delta \in \{ 0,1\}} q(z|\delta) \left|\sup_{\delta,\delta' \in \{ 0,1\}} \frac{q(z|\delta)}{q(z|\delta')} - 1 \right| \left|\bar F_{T_1}(t|x) f_{T_0}(t|x) - \bar F_{T_0}(t|x) f_{T_1}(t|x) \right|
        \\ & \le  \inf_{\delta \in \{ 0,1\}} q(z|\delta) (e^\alpha - 1 ) \left|\bar F_{T_1}(t|x) f_{T_0}(t|x) - \bar F_{T_0}(t|x) f_{T_1}(t|x) \right|.
    \end{split}
\end{equation*}
With our choice of $T_0$ and $T_1$ we have
\begin{equation*}
    \begin{split}
        \left|\bar F_{T_1}(t|x) f_{T_0}(t|x) - \bar F_{T_0}(t|x) f_{T_1}(t|x) \right| = \left| \frac{\eta_n^\beta}{t_1-t_0} I\left(\frac{\|x-x_0\|}{\eta_n} \right)\left(\int_{-\infty}^t \psi(s)\mathrm{d}s + (t-b)\psi(t) \right)\right|.
    \end{split}
\end{equation*}
Furthermore, for all $x\in\mathcal{S}_X$ and $t \in [t_0,t_1]$, we have
\begin{equation*}
    \min(m_0(t,z,x),m_1(t,z,x)) \ge  \inf_{\delta \in \{ 0,1\}} q(z|\delta) \left( \bar F_{T_1}(t|x) f_{T_0}(t) + \bar F_{T_0}(t) f_{T_1}(t|x) \right),
\end{equation*}
and 
\begin{equation*}
    \begin{split}
        \bar F_{T_1}& (t|x) f_{T_0}(t) + \bar F_{T_0}(t) f_{T_1}(t|x)
        \\ & = 2 F_{T_0}(t)f_{T_0}(t) + F_{T_0}(t) \eta^\beta_n I\left(\frac{\|x-x_0 \|}{\eta_n} \right) \psi(s) + f_{T_0}(t) \eta^\beta_n I \left( \frac{\|x-x_0 \|}{\eta_n}  \right) \int_0^t \psi (s) \mathrm{d}s
        \\ & = 2 F_{T_0}(t)f_{T_0}(t) + F_{T_0}(t) \eta^\beta_n I\left(\frac{\|x-x_0 \|}{\eta_n} \right) \psi(t) + f_{T_0}(t) \eta^\beta_n I \left( \frac{\|x-x_0 \|}{\eta_n}  \right) \int_0^t \psi (s) \mathrm{d}s
        \\ & = 2 F_{T_0}(t)f_{T_0}(t) + \eta^\beta_n I\left(\frac{\|x-x_0 \|}{\eta_n} \right) \left( F_{T_0}(t) \psi(t) + f_{T_0}(t) \int_0^t \psi (s) \mathrm{d}s \right)
        \\ & = \frac{t_1-t}{(t_1-t_0)^2} \left(2- \frac{1}{2}\eta^\beta_n I\left(\frac{\|x-x_0 \|}{\eta_n}\right)(t_1-t)(t_1-t_0) \right)
        \\ & \ge \frac{t_1-t}{(t_1-t_0)^2} \left(2- \frac{2b \kappa \sqrt{30}\|I \|_{\infty} }{I(0)  (t_1-t_0)^{3/2}}  \right).
    \end{split}
\end{equation*}
Then for $\kappa \le \frac{I(0)  (t_1-t_0)^{3/2}} {2b \sqrt{30}\|I \|_{\infty} }$ we finally get
\begin{equation*}
    \min(m_0(t,z,x),m_1(t,z,x)) \ge  \inf_{\delta \in \{ 0,1\}} q(z|\delta) \frac{t_1-t}{(t_1-t_0)^2}.
\end{equation*}
Back to \eqref{eq:DKL}, we have
\begin{equation*}
    \begin{split}
     D_{\mathrm{kl}}& (M_0 \| M_1) +D_{\mathrm{kl}}(M_1 \| M_0)
     \\ & \le (e^\alpha - 1 )^2  \int_{\R \times [t_0,t_1] \times \R^d }  \inf_{\delta \in \{ 0,1\}} q(z|\delta)\frac{(t_1-t_0)^2}{t_1-t} \left(\bar F_{T_1}(t|x) f_{T_0}(t|x) - \bar F_{T_0}(t|x) f_{T_1}(t|x) \right)^2 \mathrm{d}z \mathrm{d}t \mathrm{d}\P_X(x) 
     \\ & \le  (e^\alpha - 1 )^2 \eta_n^{2 \beta} \int_{[t_0,t_1] \times \R^d }   I\left(\frac{\|x-x_0\|}{\eta_n} \right)^2 \frac{1}{t_1-t}\left(\int_{a}^t \psi(s)\mathrm{d}s + (t-b)\psi(t) \right)^2 \mathrm{d}t \mathrm{d}\P_X(x) \int_{\R}  \inf_{\delta \in \{ 0,1\}} q(z|\delta)  \mathrm{d}z 
     \\ & \le  (e^\alpha - 1 )^2 \eta_n^{2 \beta}  \E \left[ I\left(\frac{\|X-x_0\|}{\eta_n} \right)^2 \right]  \int_{[t_0,t_1] } (t_1-t)\left(\frac{1}{2}(t-a) + \left(t- \frac{a+b}{2} \right) \right)^2 \mathrm{d}t 
     \\ & \le \frac{ 1}{16} \eta_n^{2 \beta} (e^\alpha - 1 )^2  (t_1-t_0)^4  \E \left[ I\left(\frac{\|X-x_0\|}{\eta_n} \right)^2 \right]  .
    \end{split}
\end{equation*}
With the inequality $e^x-1\le e x$ for $x\in [0,1]$ and the fact that $I$ is compactly supported on $[0,1]$ we obtain, 
\begin{equation*}
    \begin{split}
     D_{\mathrm{kl}}& (M_0 \| M_1) +D_{\mathrm{kl}}(M_1 \| M_0)
     \\ &  \le \frac{ e^2 }{16} \alpha^2 \eta_n^{2 \beta} (t_1-t_0)^4  \E \left[ \|I\|_{\infty}^2\ind_{\{ \|X-x_0\| \le \eta_n \} } \right] 
     \\ &  \le \frac{ e^2  }{16} C_X^2  \alpha^2 \eta_n^{2 \beta + \gamma} (t_1-t_0)^4  \|I\|_{\infty}^2,
    \end{split}
\end{equation*}
where we used the small ball property $\mathcal{I}_X$ in the last line. By replacing $\eta_n$ with its definition, we get
\begin{equation} \label{eq:finalDKL}
    \begin{split}
     D_{\mathrm{kl}}& (M_0 \| M_1) +D_{\mathrm{kl}}(M_1 \| M_0)
     \\ & \le \alpha^2 \kappa^{\frac{\gamma + 2\beta}{\beta}} \underbrace{ \left(\frac{4 \sqrt{30} }{I(0)  (t_1-t_0)^{5/2}} \right)^{\frac{\gamma + 2\beta}{\beta}} \frac{ e^2  }{16} C_X^2   (t_1-t_0)^4  \|I\|_{\infty}^2 }_{\eqqcolon C_{a,b,I}} ,
    \end{split}
\end{equation}
plugging \eqref{eq:finalDKL} into \eqref{eq:privateDKL}, 
\begin{equation*}
     \inf_{\tilde Q \in \mathcal{\tilde Q_\alpha}}\inf_{ \hat \Lambda} \sup_{P \in \mathcal{P}_\beta} \E_{P,\tilde Q} \left[  \left\| \hat\Lambda_n^{x_0}(M)- \Lambda_P^{x_0} \right\|^2_{D} \right] \ge  \frac{\kappa^2}{2}\left(1-\frac{1}{2}\sqrt{\kappa^{\frac{\gamma + 2\beta}{\beta}} n \alpha^2 C_{a,b,I} }\right).
\end{equation*}
Finally, the choice
\begin{equation*}
    \kappa = \left(\frac{1}{ C_{a,b,I} n \alpha^2 }\right)^{\frac{\beta}{\gamma + 2\beta}} \land \frac{I(0)  (t_1-t_0)^{3/2}} {2b \sqrt{30}\|I \|_{\infty} },
\end{equation*}
gives the expected result. \hfill\qed 

\bibliographystyle{alpha}
\bibliography{Biblio.bib}

\newcommand{\etalchar}[1]{$^{#1}$}
\begin{thebibliography}{WWZW23}

\bibitem[AG23]{Amo23}
Chiara Amorino and Arnaud Gloter.
\newblock Minimax rate for multivariate data under componentwise local
  differential privacy constraints, 2023.
\newblock Preprint in arXiv:2305.10416.

\bibitem[BAA08]{Blu08}
Ligett~K. Blum~A. and Roth A.
\newblock A learning theory approach to non-iteractive database privacy.
\newblock In {\em STOC '08: Proceedings of the fortieth annual ACM symposium on
  Theory of computing}, pages 609--618, 2008.

\bibitem[Ber81]{Ber81}
R.~Beran.
\newblock Nonparametric regression with randomly censored survival data.
\newblock {\em Tech. Rep., University of California, Berkeley}, 1981.

\bibitem[BFMSV23]{Bus23}
R\'{o}bert Busa-Fekete, Andres Mu\~{n}oz Medina, Umar Syed, and Sergei
  Vassilvitskii.
\newblock Label differential privacy and private training data release.
\newblock In {\em Proceedings of the 40th International Conference on Machine
  Learning}, ICML'23. JMLR.org, 2023.

\bibitem[BI21]{But21}
Cristina Butucea and Yann Issartel.
\newblock Locally differentially private estimation of functionals of discrete
  distributions.
\newblock In M.~Ranzato, A.~Beygelzimer, Y.~Dauphin, P.S. Liang, and J.~Wortman
  Vaughan, editors, {\em Advances in Neural Information Processing Systems},
  volume~34, pages 24753--24764. Curran Associates, Inc., 2021.

\bibitem[BM98]{art:birgeMassart_minimumconstrastestimator}
Lucien Birgé and Pascal Massart.
\newblock Minimum contrast estimators on sieves: Exponential bounds and rates
  of convergence.
\newblock {\em Bernoulli}, 4(3):329--375, 1998.

\bibitem[BNS16]{Bei16}
Amos Beimel, Kobbi Nissim, and Uri Stemmer.
\newblock Private learning and sanitization: pure vs. approximate differential
  privacy.
\newblock {\em Theory Comput.}, 12:Paper No. 1, 61, 2016.

\bibitem[BWF22a]{Bon22}
Luca Bonomi, Zeyun Wu, and Liyue Fan.
\newblock {Sharing personal ECG time-series data privately}.
\newblock {\em Journal of the American Medical Informatics Association},
  29(7):1152--1160, 04 2022.

\bibitem[BWF22b]{Bon22a}
Luca Bonomi, Zeyun Wu, and Liyue Fan.
\newblock {Sharing personal ECG time-series data privately}.
\newblock {\em Journal of the American Medical Informatics Association},
  29(7):1152--1160, 04 2022.

\bibitem[CH11]{Cha11}
Kamalika Chaudhuri and Daniel Hsu.
\newblock Sample complexity bounds for differentially private learning.
\newblock In Sham~M. Kakade and Ulrike von Luxburg, editors, {\em Proceedings
  of the 24th Annual Conference on Learning Theory}, volume~19 of {\em
  Proceedings of Machine Learning Research}, pages 155--186, Budapest, Hungary,
  09--11 Jun 2011. PMLR.

\bibitem[CR14]{art:ChagnyRoche_AdaptiveMinimax}
Ga\"{e}lle Chagny and Angelina Roche.
\newblock Adaptive and minimax estimation of the cumulative distribution
  function given a functional covariate.
\newblock {\em Electron. J. Stat.}, 8(2):2352--2404, 2014.

\bibitem[DJW18]{art:Duchi_Jordan_Wainwright_MinimaxLocalPrivat}
John~C. Duchi, Michael~I. Jordan, and Martin~J. Wainwright.
\newblock Minimax optimal procedures for locally private estimation.
\newblock {\em J. Amer. Statist. Assoc.}, 113(521):182--201, 2018.

\bibitem[DR13]{Dwo13}
Cynthia Dwork and Aaron Roth.
\newblock The algorithmic foundations of differential privacy.
\newblock {\em Found. Trends Theor. Comput. Sci.}, 9(3-4):211--487, 2013.

\bibitem[DS10]{Dwo10}
Cynthia Dwork and Adam Smith.
\newblock Differential privacy for statistics: What we know and what we want to
  learn.
\newblock {\em Journal of Privacy and Confidentiality}, 1(2), Apr. 2010.

\bibitem[EBG23]{Esc23-3}
M.~Escobar-Bach and O.~Goudet.
\newblock Survival estimation for missing not at random censoring indicators
  based on copula models, 2023+.
\newblock Preprint in arXiv:2009.01726.

\bibitem[EBGG18]{art:EscoGoegebeurGuillouLocal}
Mikael Escobar-Bach, Yuri Goegebeur, and Armelle Guillou.
\newblock Local robust estimation of the {P}ickands dependence function.
\newblock {\em Ann. Statist.}, 46(6A):2806--2843, 2018.

\bibitem[FV06]{ferraty2006nonparametric}
Fr{\'e}d{\'e}ric Ferraty and Philippe Vieu.
\newblock {\em Nonparametric functional data analysis: theory and practice},
  volume~76.
\newblock Springer, 2006.

\bibitem[FWC{\etalchar{+}}21]{Fic21}
Joseph Ficek, Wei Wang, Henian Chen, Getachew Dagne, and Ellen Daley.
\newblock {Differential privacy in health research: A scoping review}.
\newblock {\em Journal of the American Medical Informatics Association},
  28(10):2269--2276, 08 2021.

\bibitem[GG02]{art:GineGuillou2002Rates}
Evarist Gin\'{e} and Armelle Guillou.
\newblock Rates of strong uniform consistency for multivariate kernel density
  estimators.
\newblock {\em Ann. Inst. H. Poincar\'{e} Probab. Statist.}, 38(6):907--921,
  2002.
\newblock En l'honneur de J. Bretagnolle, D. Dacunha-Castelle, I. Ibragimov.

\bibitem[JCDW18]{Duc18}
Michael I.~Jordan John C.~Duchi and Martin~J. Wainwright.
\newblock Minimax optimal procedures for locally private estimation.
\newblock {\em Journal of the American Statistical Association},
  113(521):182--201, 2018.

\bibitem[LO19]{Liu19}
Xiyang Liu and Sewoong Oh.
\newblock Minimax rates of estimating approximate differential privacy, 2019.
\newblock Preprint in arXiv:1905.10335.

\bibitem[LS86]{art:lo1986product}
Shaw-Hwa Lo and Kesar Singh.
\newblock The product-limit estimator and the bootstrap: some asymptotic
  representations.
\newblock {\em Probability Theory and Related Fields}, 71(3):455--465, 1986.

\bibitem[Nar23]{Nar23}
Shyam Narayanan.
\newblock Better and simpler lower bounds for differentially private
  statistical estimation, 2023.
\newblock Preprint in arXiv:2310.06289.

\bibitem[NH17]{Ngu17}
Th\^{o}ng~T. Nguy\^{e}n and Siu~Cheung Hui.
\newblock Differentially private regression for discrete-time survival
  analysis.
\newblock In {\em Proceedings of the 2017 ACM on Conference on Information and
  Knowledge Management}, CIKM '17, page 1199–1208, New York, NY, USA, 2017.
  Association for Computing Machinery.

\bibitem[OSML12]{OKe12}
Christine~M. O’Keefe, Ross~Stewart Sparks, Damien McAullay, and Bronwyn
  Loong.
\newblock Confidentialising survival analysis output in a remote data access
  system.
\newblock {\em Journal of Privacy and Confidentiality}, 4(1), Jul. 2012.

\bibitem[Tsy09]{book:Tsybakov_nonparam}
Alexandre~B. Tsybakov.
\newblock {\em Introduction to nonparametric estimation}.
\newblock Springer Series in Statistics. Springer, New York, 2009.
\newblock Revised and extended from the 2004 French original, Translated by
  Vladimir Zaiats.

\bibitem[VKV97]{art:VanKeilegom1997Estim}
Ingrid Van~Keilegom and No\"{e}l Veraverbeke.
\newblock Estimation and bootstrap with censored data in fixed design
  nonparametric regression.
\newblock {\em Ann. Inst. Statist. Math.}, 49(3):467--491, 1997.

\bibitem[Was12]{Was12}
Larry Wasserman.
\newblock Minimaxity, statistical thinking and differential privacy.
\newblock {\em Journal of Privacy and Confidentiality}, 4(1), Jul. 2012.

\bibitem[WWZW23]{Wan23}
Yanling Wang, Qian Wang, Lingchen Zhao, and Cong Wang.
\newblock Differential privacy in deep learning: Privacy and beyond.
\newblock {\em Future Generation Computer Systems}, 148:408--424, 2023.

\bibitem[WX19]{Wan19}
Di~Wang and Jinhui Xu.
\newblock On sparse linear regression in the local differential privacy model.
\newblock In Kamalika Chaudhuri and Ruslan Salakhutdinov, editors, {\em
  Proceedings of the 36th International Conference on Machine Learning},
  volume~97 of {\em Proceedings of Machine Learning Research}, pages
  6628--6637. PMLR, 09--15 Jun 2019.

\bibitem[WZ10]{Was10}
Larry Wasserman and Shuheng Zhou.
\newblock A statistical framework for differential privacy.
\newblock {\em Journal of the American Statistical Association},
  105(489):375--389, 2010.

\end{thebibliography}

\end{document}